\documentclass[UKenglish]{lmcs}
\pdfoutput=1

\usepackage{lastpage}
\lmcsdoi{18}{2}{11}
\lmcsheading{}{\pageref{LastPage}}{}{}%
{Dec.~22,~2020}{May~25,~2022}{}

\pdfoutput=1

\usepackage[utf8]{inputenc}
\usepackage[UKenglish]{babel} 
 
\usepackage[babel]{csquotes} 
\usepackage{hyperref} 
\usepackage{url}
\usepackage[yyyymmdd,hhmmss]{datetime}

\usepackage{enumitem}  

\usepackage{amsmath}
\usepackage{amsfonts}
 \usepackage{amssymb}
\usepackage{graphicx}
\usepackage{amsthm}
\usepackage{thmtools,thm-restate}  
\usepackage[ruled]{algorithm2e} % algorithms
 \SetKw{KwInput}{Input:} 
 \SetKw{KwQuestion}{Question:}  
\usepackage{mathtools} % for mathclap

\DeclareMathOperator{\Betw}{Betw}
\DeclareMathOperator{\Cycl}{Cycl}
\DeclareMathOperator{\Sep}{Sep}

\DeclareMathOperator{\dll}{dual-ll}
\renewcommand{\ll}{\textup{\textup{ll}}}
\DeclareMathOperator{\lex}{lex}
\DeclareMathOperator{\mx}{mx}

\DeclareMathOperator{\mi}{mi}

\DeclareMathOperator{\mix}{mix}

\DeclareMathOperator{\pp}{pp}
\DeclareMathOperator{\dpp}{dual-pp}
\DeclareMathOperator{\Aut}{Aut}
\DeclareMathOperator{\End}{End}

\DeclareMathOperator{\Th}{Th}

\DeclareMathOperator{\Inv}{Inv}

\DeclareMathOperator{\Pol}{Pol}
\DeclareMathOperator{\GF}{GF}

\newcommand{\CSP}{\mathrm{CSP}}
\newcommand{\ceq}{\coloneqq}

\newcommand{\J}{R^{\mix}}
\newcommand{\Jtext}{Rmix}

\newcommand{\AND}{\wedge}
\newcommand{\OR}{\vee}

\newcommand{\Q}{\mathbb{Q}}
\newcommand{\N}{\mathbb{N}}
\newcommand{\Z}{\mathbb{Z}}

\newcommand{\fA}{\mathfrak{A}}
\newcommand{\fB}{\mathfrak{B}}

\newcommand{\fM}{\mathfrak{M}}

\newcommand{\cF}{\mathcal{F}}
\newcommand{\cK}{{\mathcal{K}}}

\newcommand{\set}[1]{\!\left\{ #1 \right\}}

\usepackage{tikz}  

\newcommand{\Fresse}{Fra\"{i}ss\'{e}}

\newlength{\myl}

\newcommand\cxmarker[3] { %{x}{y}{label}
  \draw (#1, #2+0.1) -- (#1, #2-0.1);
  \node at (#1, #2 - 0.4){#3};
}
\newcommand\cymarker[3] { %{x}{y}{label}
  \draw (#1+0.1, #2) -- (#1-0.1, #2);
  \node at (#1-0.4, #2){#3};
}

\title{Tractable Combinations of Temporal CSPs}

\author[M.~Bodirsky]{Manuel Bodirsky\rsuper{a}}
\author[J.~Greiner]{Johannes Greiner\rsuper{a}}
\address{Institut f\"ur Algebra, Technische Universit\"at Dresden,  01062 Dresden, Germany} 
\email{\texttt{manuel.bodirsky@tu-dresden.de}, \texttt{contact.johannes.greiner@gmail.com}}
\author[J.~Rydval]{Jakub Rydval\rsuper{b}}
\address{Institut f\"ur Theoretische Informatik, Technische Universit\"at Dresden,  01062 Dresden, Germany} 
\email{\texttt{jakub.rydval@tu-dresden.de}}
\thanks{All authors have received funding from the European Research Council (ERC Grant Agreement no. 681988, CSP-Infinity) and the DFG Graduiertenkolleg 1763 (QuantLA)}

\begin{document}

\begin{abstract}
  The constraint satisfaction problem (CSP) of a first-order theory $T$ is the computational problem of deciding whether a given conjunction of atomic formulas is satisfiable in some model of $T$. 
We study the computational complexity of
$\CSP(T_1 \cup T_2)$ where $T_1$ and $T_2$ are theories with disjoint finite relational signatures. 
We prove that if $T_1$ and $T_2$ are the theories of temporal structures, i.e., structures where all relations have a first-order definition in $(\Q;<)$, then $\CSP(T_1 \cup T_2)$ is in P or NP-complete. 
To this end we prove a purely algebraic statement about the structure of the lattice of locally closed clones over the domain ${\mathbb Q}$ that contain $\Aut(\Q;<)$.
\end{abstract}	

% ASCii
% The constraint satisfaction problem (CSP) of a first-order theory T is the computational problem of deciding whether a given conjunction of atomic formulas is satisfiable in some model of T. We study the computational complexity of CSP$(T_1 \cup T_2)$ where $T_1$ and $T_2$ are theories with disjoint finite relational signatures. We prove that if $T_1$ and $T_2$ are the theories of temporal structures, i.e., structures where all relations have a first-order definition in $(Q;<)$, then CSP$(T_1 \cup T_2)$ is in P or NP-complete.  To this end we prove a purely algebraic statement about the structure of the lattice of locally closed clones over the domain $Q$ that contain Aut$(Q;<)$.

\maketitle

%\medskip
%\noindent \textbf{Keywords:} temporal structures, omega-categoricity, constraint satisfaction problems, combination, complexity, definability, universal algebra\\
%\noindent \textbf{2010 Mathematics Subject Classification:} 06A05, 68Q25, 08A70

\section{Introduction}\label{sec:introduction}
Deciding the satisfiability of formulas with respect to a given theory or structure is one of the fundamental problems in theoretical computer science. One large class of problems of this kind are \emph{Constraint Satisfaction Problems} (CSPs). For a finite relational signature $\tau$, the CSP of a $\tau$-theory $T$, written $\CSP(T)$, is the computational problem of deciding whether a given finite set $S$ of atomic $\tau$-formulas is satisfiable in some model of $T$.
A general goal is to identify theories $T$ such that $\CSP(T)$ can be solved in polynomial time.

Many theories that are relevant in program verification and automated deduction are of the form $T_1 \cup T_2$ where the signatures of $T_1$ and $T_2$ are disjoint; satisfiability problems of the form $\CSP(T_1 \cup T_2)$ 
are also studied in the field of Satisfiability Modulo Theories (SMT). 
If we already have a decision procedure for $\CSP(T_1)$ and for $\CSP(T_2)$, then, under certain conditions, we can use these decision procedures to construct a decision procedure for $\CSP(T_1 \cup T_2)$ in a generic way. Most results in the area of combinations of decision procedures concern decidability, rather than polynomial-time decidability; see for example~\cite{Ghilardi05_ModelTheoreticMethodsInCombinedConstraintSatisfiability,TinelliRingeissen03_UnionsOfNonDisjointTheoriesAndCombinationsOfSatisfiabilityProcedures,
BonacinaGhilardiEA06_DecidabilityAndUndecidabilityResultsForNelsonOppenAndRewriteBasedDecisionProcedures,
Ringeissen96_CooperationOfDecisionProceduresForTheSatisfiabilityProblem}.
We are particularly interested in polynomial-time decidability  
and the borderline to NP-hardness. 
The seminal result in this direction is due to Greg Nelson and Derek C. Oppen, who provided a criterion assuring that satisfiability of conjunctions of atomic and negated atomic formulas can be decided in polynomial time~\cite{NelsonOppen79, Oppen80_ComplexityConvexityAndCombinationsOfTheories}.
The work of Nelson and Oppen has been further developed later on (see for example~\cite{BaaderSchulz}) and their algorithm has been implemented in many SMT solvers (see for example~\cite{KrsticGoel07_ArchitectingSolversForSatModuloTheoriesNelsonOppenWithDpll}). 
%,\cite{DeMoura07_DevelopingEfficientSmtSolvers}). 
%
While their result directly gives sufficient conditions for
polynomial-time tractability of $\CSP(T_1 \cup T_2)$, one of their
conditions called \enquote{convexity} can be weakened to \enquote{independence of $\neq$} (see~\cite{BroxvallJonssonRenz})
without changing their proof, if we only consider conjunctions of
atomic formulas as input (see Section~\ref{sec:NO} and Section~\ref{sec:ComplexityResults} for details).
Interestingly, the weakened criterion also turns out to be remarkably tight;
Schulz~\cite{Schulz00_WhyCombinedDecisionProblemsAreOftenIntractable} as well as Bodirsky and Greiner~\cite{BodirskyGreinerCombinations} proved that in many cases not covered by the weaker criterion, $\CSP(T_1 \cup T_2)$ is NP-hard even though both $\CSP(T_1)$ and $\CSP(T_2)$ can be solved in polynomial time. 
However, there are examples of theories $T_1$ and $T_2$ that do not satisfy the weakened conditions of Nelson and Oppen, but $\CSP(T_1 \cup T_2)$ can be solved in polynomial time nevertheless (see~\cite{BodirskyGreiner_Sampling}). 
Unfortunately, there is still no general theory of polynomial-time tractability for combinations of theories. 

An important subclass of constraint satisfaction problems are \emph{temporal CSPs}, which are CSPs for the theories of structures of the form $({\mathbb Q};R_1,\dots,R_n)$ where $R_1,\dots,R_n$ are relations defined by quantifier-free first-order formulas over $({\mathbb Q};<)$; we refer to such structures as \emph{temporal structures}. 
A well-known example of such a structure is $({\mathbb Q};\Betw)$ where $\Betw := \{(a,b,c) \mid a<b<c \vee c<b<a\}$. 
The CSP for the theory of this structure is the so-called \emph{Betweenness problem} and is NP-complete~\cite{Opatrny}. 
Other well-known temporal CSPs are the Cyclic Ordering problem~\cite{GalilMegiddo}, Ord-Horn constraints~\cite{Nebel}, the network satisfaction problem for the point algebra~\cite{PointAlgebra}, and scheduling with and/or precedence constraints~\cite{and-or-scheduling}. 
It has been shown that every temporal CSP is in P or NP-complete~\cite{tcsps-journal}. 
Temporal CSPs are of particular importance for the study of polynomial-time procedures for combinations of theories, because many of the polynomial-time tractable cases do \emph{not} satisfy the weakened conditions of Nelson and Oppen because $\neq$ is not independent in these cases. 
This is unlike several other classifications for CSPs where all the polynomial-time tractable cases do satisfy the weakened conditions of Nelson and Oppen~\cite{ecsps, BMPP16, equiv-csps, BodPin-Schaefer-both, BodHils, posetCSP18,MMSNP-Journal} and hence CSPs for combinations of such theories can be solved in polynomial time. 
Some results about the complexity of CSPs for combinations of theories of temporal structures
were obtained in~\cite{BodirskyGreinerCombinations},
but they were restricted to temporal structures that contain the relations $<$ and $\neq$.

%%%%%%%%%%%%%%%%%%%%%%%%%%%%%%%%%%%%%%%%%%%%%%%%%%%%%%%%%%%%%%%%%%%%%%%%%%%%
\subsection{Contributions}
\label{sec:contributions}

Our main result is a complexity dichotomy for all problems of the form $\CSP(T_1 \cup T_2)$ where $T_1$ and $T_2$ are first-order theories of temporal structures with disjoint finite signatures. 
In order to phrase our results in this section, we need the concepts of \emph{primitive positive definability} and \emph{polymorphisms}, which are of fundamental importance in universal algebra and will be recalled in Section~\ref{sec:UA}.
The main result is the following:

\begin{thm}\label{thm:dichotomyExpansionsQ}
Let $T_1$ and $T_2$ be the theories of temporal structures $\fA_1$ and $\fA_2$ with disjoint finite signatures. 
%If one of $\fA_1,\fA_2$ is preserved by all permutations
Then $\CSP(T_1 \cup T_2)$ is polynomial-time tractable if 
\begin{enumerate}
\item for both $i \in \{1,2\}$, the structure $\fA_i$ has a binary injective polymorphism and $\CSP(\fA_i)$ is in P, or 
\item for both $i \in \{1,2\}$, the structure $\fA_i$ has a constant polymorphism, or 
\item there is a temporal structure $\fB$ such that $\CSP(\fB) = \CSP(T_1 \cup T_2)$, and $\CSP(\fB)$ is in P (this happens if, for some $i \in \{1,2\}$, all permutations are polymorphisms of 
$\fA_i$). 
\end{enumerate}
Otherwise, $\CSP(T_1 \cup T_2)$ is  NP-complete.
\end{thm}

The technique we use to prove NP-hardness in Theorem~\ref{thm:dichotomyExpansionsQ} is based on the notion of cross prevention  introduced in~\cite{BodirskyGreinerCombinations}.

\begin{defi}\label{def:preventCrosses}
  A $\tau$-structure $\fB$ \emph{can prevent crosses} if there exists a primitive positive $\tau$-formula $\phi(x,y,u,v)$ such that
  \begin{enumerate}
   \item $\phi(x,y,u,v) \AND x=y \AND u \neq v \AND x\neq u \AND x\neq v$ is satisfiable in $\fB$,
  \item $\phi(x,y,u,v) \AND x\neq y \AND u=v  \AND x\neq u \AND y\neq u $ is satisfiable in $\fB$, and
  \item $\phi(x,y,u,v) \AND x=y \AND u=v$ is not satisfiable in $\fB$. 
  \end{enumerate}
  Any such formula $\phi$ will be referred to as a \emph{cross prevention formula} of $\fB$.
\end{defi}

An example of a structure that can prevent crosses is $(\Q; <)$; a cross prevention formula is $u<x \wedge y<v$. Another example is $(\N; E, N)$ where $E$ is an equivalence relation where all classes have exactly two elements and $N$ is the complement of $E$. In this structure $E(x,u) \AND N(y,v)$ is a cross prevention formula.

Our next contribution, Theorem~\ref{thm:JPreventCrossesHard}, is the complexity result underlying the hardness proof for Theorem~\ref{thm:dichotomyExpansionsQ} and is not limited to temporal structures. It uses the relation $\J$, which is of fundamental importance to this article and defined as follows:
\begin{align*}
 \J  \ceq \; & \big\{ (a_1,a_2,a_3)\in \mathbb{Q}^{3}\mid (a_1=a_2) \vee (a_3<a_1 \wedge a_3< a_2) \big\} \\
= \; & \big \{(a_1,a_2,a_3) \in {\mathbb Q}^n \mid a_3 \geq \min(a_1,a_2) \Rightarrow a_1=a_2 \big \} .
\end{align*}

\begin{thm}\label{thm:JPreventCrossesHard}
  Let $\fA$ be a countably infinite $\omega$-categorical structure with finite relational signature and without algebraicity. If $\fA$ can prevent crosses, then 
 $\CSP(\Th({\mathbb Q}; <, \J) \cup  \Th(\fA))$ is NP-hard.
\end{thm}

Examples of $\omega$-categorical structures without algebraicity and with cross prevention can be found in Section~\ref{sec:combine}.

Our third contribution is the algebraic cornerstone of this article, which is a result about the definability  of $\J$. 
If $R$ is a temporal relation, then $-R$ denotes the \emph{dual of $R$}, which is the  temporal relation $\{(a_1,\dots,a_n) \in \Q^n \mid (-a_1,\dots,-a_n) \in R\}$. 
The \emph{dual} of an operation $f\colon {\mathbb Q}^n \to {\mathbb Q}$ 
 is defined by $(x_{1}, \dots ,x_{n}) \mapsto -f(-x_{1}, \dots , -x_{n})$. Hence, for any temporal relation $R$ and any operation $f$ on $\Q$,
the operation $f$ preserves $R$ if and only if the 
dual of $f$ preserves the dual of $R$.
The functions $\min, \min, \mx$ and $\ll$ will be explained in Section~\ref{sec:TCLs}.

 \begin{thm}\label{thm:BinInjOrJ}
  Let $\fA$ be a first-order expansion of $(\Q;<)$ with a finite relational signature such that 
  %$\fA$ is preserved by 
  $\min$, $\mi$, $\mx$, $\ll$ or one of their duals is a polymorphism of $\fA$. Then the following are equivalent:
  \begin{itemize}
  %\item $\ll$ or $\dll$ is a polymorphism of $\fA$,
  \item $\fA$ does not have a binary injective polymorphism.
  \item $\J$ or its dual $-\J$ has a primitive positive definition in $\fA$. 
  \end{itemize}
\end{thm}
 
Theorem~\ref{thm:BinInjOrJ} characterises the first-order expansions of $(\Q;<)$ among the polynomial-time tractable cases in the dichotomy of Bodirsky and K\'ara (see Theorem~\ref{thm:bod-kara}) whose first-order theory does not satisfy the weakened tractability conditions by Nelson and Oppen because $\neq$ is not independent from their theory (see Section~\ref{sec:NO} for the definition).

\subsection{Significance of the Result in Universal Algebra}
Theorem~\ref{thm:BinInjOrJ} is of independent interest in universal algebra; for an introduction to the universal-algebraic concepts that appear in this section we refer the reader to Section~\ref{sec:UA}. 
Theorem~\ref{thm:BinInjOrJ} can be seen as a result about locally closed clones on a countably infinite domain $B$ that are \emph{highly set-transitive}. 
A permutation group $G$ on a set $B$ is said to be highly set-transitive if for all finite subsets $S_1$ and $S_2 $ of $B$ of equal size there exists a permutation in $G$ that maps $S_1$ to $S_2$. An operation clone on a set $B$ is said to be highly set-transitive if it contains a highly set-transitive permutation group. 

It can be shown that the highly set-transitive locally closed clones are precisely the polymorphism clones of temporal structures (possibly with infinitely many relations), up to a bijection between $B$ and $\Q$~\cite{tcsps-journal}. 
These objects form a lattice: the meet of two clones is the intersection of the clones and the join can be obtained as the polymorphism clone of 
all relations preserved by both of the clones (see, e.g., Section 6.1 in~\cite{Book}). 
Similarly, as the lattice of clones over the
set $\{0,1\}$ plays a fundamental role for studying finite algebras (it has been classified by Post~\cite{Post}), the lattice of locally closed highly set-transitive clones over $\Q$ is of fundamental importance for the study of locally closed clones in general. 
This lattice is of size $2^{\omega}$ even if we restrict our attention to closed clones that contain 
all permutations~\cite{BodChenPinsker}. 
However, the lower parts of the lattice appears to be more structured and amenable to classification.
We pose the following question.

\begin{qu}\label{quest:lattice}
Are there only countably many locally closed highly set-transitive clones over a fixed countably infinite set that do not contain a binary injective operation? 
\end{qu}

Question~\ref{quest:lattice} has a positive answer
in the case that the clone contains all permutations of the base set~\cite{BodChenPinsker}. 
Theorem~\ref{thm:end} below shows that answering 
Question~\ref{quest:lattice} can be split into finitely many cases, depending on whether the clone contains a constant operation, or whether it preserves one out of a finite list of temporal relations. 
%We believe that the case where the clones preserve $<$ is the most interesting case. 
Theorem~\ref{thm:BinInjOrJ} shows that in case~\ref{case:tractableCasesTCL} of Theorem~\ref{thm:bod-kara}, we can even focus on clones that preserve the relation $\J$ or its dual.

%%%%%%%%%%%%%%%%%%%%%%%%%%%%%%%%%%%%%%%%%%%%%%%%%%%%%%%%%%%%%%%%%%%%%%%%%%%%

\subsection{Outline of the Article}
We first recall some basic concepts from model theory in Section~\ref{sec:mt}.
Then, the classical Nelson-Oppen conditions for obtaining polynomial-time decision procedures for combined theories are presented in Section~\ref{sec:NO}; 
a slight generalisation of their results can be found in Section~\ref{sec:ComplexityResults}. 
We then define the model-theoretic notion of a \emph{generic combination} of two structures with disjoint relational signatures in Section~\ref{sec:genericCombinations}, which plays a crucial role in our proof. 
The reason is that we may apply universal algebra to study the complexity of CSPs of structures but not of theories.
Basic universal-algebraic concepts are introduced in Section~\ref{sec:UA}. 
Our results build on the classification of the temporal CSPs that can be solved in polynomial time~\cite{tcsps-journal}, which we present along with other known facts about temporal structures in Section~\ref{sec:TCLs}. 

The proof of Theorem~\ref{thm:BinInjOrJ} is organised as follows. 
The difficult direction is to find a primitive positive definition of $\J$ 
%or $-\J$ 
in $\fA$ if $\fA$ is not preserved by a binary injective polymorphism. If $\Pol(\fA)$ contains $\mi$, then 
the proof is easier if $\leq$ is primitively positively definable in the structure $\fA$. 
If the relation $\leq$ is not primitively positively definable in $\fA$, then a certain operation $\mix$ is a polymorphism of $\fA$. 
We discuss $\mix$ in Section~\ref{sec:mix} and use results thereof in Section~\ref{sec:mi} to show the primitive positive definability of $\J$ in $\fA$.

The case that $\Pol(\fA)$ contains $\mx$ but not $\mi$ is treated in Section~\ref{sec:mx}, 
and the case that $\Pol(\fA)$ contains $\min$ but neither $\mi$ nor $\mx$ is treated in Section~\ref{sec:min}. All of these partial results are put together in Section~\ref{sec:dicho}. 

Finally, Section~\ref{sec:combine} uses our definability dichotomy theorem (Theorem~\ref{thm:BinInjOrJ})
to prove the complexity dichotomy for combinations of temporal CSPs.  

\section{Preliminaries}\label{sec:preliminaries}
We use the notation $[k]$ for the set $\set{1,\dots, k}\subseteq \N$.  

\subsection{Model Theory}
\label{sec:mt}
A \emph{relational signature} is a set of relation symbols, each endowed with a natural number, stating its arity.
Let $\tau$ be relational signature. A \emph{$\tau$-structure} $\fA$ consists of a set $A$, the \emph{domain} of $\fA$, and a relation $R \subseteq A^k$ for each $R\in \tau$ of arity $k$. We use the notation $\fA = (A; R_1, \dots, R_n)$ for relational structures with finite signature.

A $\tau$-formula is \emph{atomic} if it is of the form $x_1 = x_2$, $\bot$ (the logical “false”), or
$R(x_1, \dots, x_n)$ for $R\in \tau$ of arity $n$ where $x_1, \dots,  x_n$ are variables.
A \emph{literal} is either an atomic formula or a negated atomic formula. 
A $\tau$-formula is \emph{primitive positive (pp)} if it is of the form $\exists x_{k}, x_{k+1},\dots, x_{\ell}\ldotp \phi(x_1,\dots,x_{\ell})$ where $\phi$ is a conjunction of atomic $\tau$-formulas and $k\geq 1$ is allowed to be larger than $\ell$, in which case all variables are unquantified.
A $\tau$-formula is \emph{existential positive} if it 
is a disjunction of primitive positive formulas; note that every first-order formula which does not contain  negation or universal quantification is equivalent to such a formula.
A \emph{$\tau$-theory} is a set of first-order $\tau$-sentences, i.e.,  $\tau$-formulas without free variables.  
For a $\tau$-strucutre $\fA$ the \emph{(first-order) theory of $\fA$}, denoted by $\Th(\fA)$, is the set of all first-order $\tau$-sentences that hold in $\fA$. 
If $T$ is a $\tau$-theory and $\fA$ a $\tau$-structure, then $\fA$ is a \emph{model} for $T$, written $\fA \models T$, if all sentences in $T$ hold in $\fA$. In particular $\fA \models \Th(\fA)$.

The CSP of a $\tau$-structure $\fA$, written $\CSP(\fA)$, is the computational problem of deciding, given a conjunction of atomic $\tau$-formulas, whether or not the conjunction is satisfiable in $\fA$. 
More generally, the CSP of a $\tau$-theory $T$, written $\CSP(T)$, is the computational problem of deciding whether a given conjunction of atomic $\tau$-formulas is satisfiable in \emph{some} model of $T$. Note that $\CSP(\fA)$ and $\CSP(\Th(\fA))$ are the same problem. 
Let $\fA$ be a relational $\tau$-structure and $\fB$ a relational $\sigma$-structure with $\tau\subseteq \sigma$.

If $\fA$ is a $\tau$-structure and $\phi(x_1,\dots,x_n)$ is a $\tau$-formula with free variables $x_1,\dots,x_n$, then the 
\emph{relation defined by $\phi$} is the relation 
$\set{(a_1, \dots, a_n)\in A^n \mid \fA \models \phi(a_1, \dots, a_n)}$.
We say that a relation is \emph{primitively positively definable} in $\fA$ if 
there is a primitive positive formula that defines $R$ in $\fA$.  
%By 
%$\langle \fA \rangle$ 
%we denote the set of all relations which are primitively positively definable in $\fA$. 
First-order and existential positive definability are defined analogously. 
Notice that a definition of a relation $R$ via a formula $\phi$ in the above way also yields a bijection between coordinates of tuples of $R$ and the free variables of $\phi$. We will use this bijection implicitly whenever we say that $t\in R$ \emph{satisfies} a formula on the free variables of $\phi$.

If $\fA$ can be obtained from $\fB$ by deleting relations from $\fB$, then $\fA$ is called a \emph{reduct} of $\fB$, and $\fB$ is called an \emph{expansion} of $\fA$. 
If the signature of $\fA$ equals $\tau$, then the reduct $\fA$ of $\fB$ is also denoted by $\fB^{\tau}$.  
An expansion $\fB$ of $\fA$ is called \emph{first-order expansion} if all relations in $\fB$ have a first-order definition in $\fA$. The expansion of $\fA$ by a relation $R$ is denoted by $(\fA;R)$.
As usual, $\Aut(\fA)$ denotes the set of all automorphisms of $\fA$, i.e., isomorphisms from $\fA$ to $\fA$.
 For $k\in \N$ and $a\in A^k$, the set $\Aut(\fA)a \ceq \set{(\alpha(a_1), \dots, \alpha(a_k)) \mid \alpha \in \Aut(\fA)}$ is called the \emph{orbit} of $a$.

The theory of $(\Q;<)$, or any first-order expansion thereof, has the remarkable property of \emph{$\omega$-categoricity}, that is, it has only one countable model up to isomorphism (see, e.g.,~\cite{Hodges}). 
% (see Proposition~3.1.1 in~\cite{Bodirsky-HDR-v8}). 
The class of $\omega$-categorical  relational structures can be characterised by the following theorem.

\begin{thmC}[Engeler, Ryll-Nardzewski, Svenonius, see~\cite{Hodges}, p.~171]\label{thm:ryll} %\cite{HodgesLong} p. 341
  Let $\fA$ be a countably infinite structure with countable signature. Then, the following are equivalent:
  \begin{enumerate}
  \item $\fA$ is $\omega$-categorical;
  \item for all $n\geq 1$ every orbit of $n$-tuples is first-order definable in $\fA$;
  \item for all $n\geq 1$ there are only finitely many orbits of $n$-tuples.
  \end{enumerate}
\end{thmC}

\subsection{Universal Algebra}\label{sec:UA}

 A operation $f\colon A^m \rightarrow A$ \emph{preserves} a relation $R \subseteq A^n$ if for all $t_1, \dots, t_m \in R$ we have $f(t_1, \dots, t_m) \in R$ where $f$ is applied component-wise.  
 For instance, the projection of arity $n$ to the $i$-th coordinate, denoted by $\pi^n_i$, preserves every relation over $A$. 
 For a set $S$ of relations over $A$ we define $\Pol(S)$ as the set of all operations on $A$ that preserve all relations in $S$. We define $\Pol(\fA)$ as $\Pol(S)$ where $S$ is the set of all relations of $\fA$. Unary polymorphisms are also called \emph{endomorphisms} of $\fA$; the set of all endomorphisms is denoted by $\End(\fA)$.  

For a set $S$ of functions on a set $A$ we define $\Inv(S)$ (`invariants of $S$') as the set of all finitary relations over $A$ which are preserved by all functions in $S$. 

\begin{thmC}[\cite{BodirskyNesetrilJLC}, Theorem~4] \label{thm:InvPol} Let $\fA$ be a countable $\omega$-categorical relational structure. Then a relation $R$ over $A$ is preserved by the polymorphisms of $\fA$ if and only if $R$ has a primitive positive definition in $\fA$.
%, i.e., $\Inv(\Pol(\fA))=\langle \fA \rangle$.  
\end{thmC}

As a consequence of Theorem~\ref{thm:InvPol}, we may go back and forth between the existence of certain polymorphisms and the primitive positive definability of certain relations. Furthermore, Theorem~\ref{thm:InvPol} implies that the set of polymorphisms of an $\omega$-categorical relational structure $\fA$ fully captures the complexity of $\CSP(\fA)$.

%,

%% eigentlich reden wir hier nur über operation clones. 
One of the central notions of universal algebra is that of a clone. A set of operations on a common domain is a \emph{clone} if it contains all projections and is closed under composition of functions. Thus, if we fix the domain, an arbitrary intersection of clones is again a clone. Therefore, given a set of operations $F$ over a common domain, there is a unique minimal clone $\langle F \rangle$ containing $F$, which we call the clone \emph{generated} by $F$. For a clone $\cF$ on domain $A$ we will also need the \emph{local closure} of $\cF$, denoted by $\overline{\cF}$, which is the smallest clone which contains $\cF$ and for any $n\in \N$ and $g\colon A^n \rightarrow A$ the following holds: If for all finite $S\subseteq A$ there exists $f_S \in \cF$ such that $f_S|_{S^n} = g|_{S^n}$ then $g\in \overline{\cF}$. 
If $\cF = \overline{\cF}$, then $\cF$ is \emph{locally closed}. It is easy to show that $\Pol(\fA)$ is always a locally closed clone for any relational structure $\fA$.

\subsection{The Conditions of Nelson and Oppen}
\label{sec:NO}
In this section we recall the classical conditions of Nelson and Oppen on theories $T_1$ and $T_2$ with disjoint signatures that guarantee the polynomial-time tractability of $\CSP(T_1 \cup T_2)$. Their condition can be found in~\cite{NelsonOppen79, Oppen80_ComplexityConvexityAndCombinationsOfTheories} and~\cite{BaaderSchulz} and are the following:

\begin{itemize}
 \item Both theories $T_1$ and $T_2$ must be \emph{stably infinite}, i.e., whenever a finite set of literals $S$ is satisfiable in a model of the theory, then there is also an infinite model of the theory where $S$ is satisfiable. 
 \item Both theories must be \emph{convex}, i.e., if we choose a finite set of literals $S$ such that for all $i\in [n]$ there exist a model of the theory where $S \cup \set{x_i \neq y_i}$ is satisfiable, then there exists a model of the theory where $S \cup \set{x_1\neq y_1, \dots, x_n \neq y_n}$ is satisfiable.
 \item For $i=1$ and $i=2$ there exist polynomial-time decision procedures to decide whether a finite set of $\tau_i$-literals is satisfiable in some model of $T_i$.
\end{itemize}

The theorem of Nelson and Oppen states that if $T_1$ and $T_2$ satisfy these three conditions, then there exists a polynomial-time procedure 
that decides whether a given set of literals over the signatures of $T_1$ and $T_2$ is satisfiable in a model of $T_1 \cup T_2$.
Note that this decision problem is in general not equal to $\CSP(T_1 \cup T_2)$, as $S$ is restricted to  atomic formulas in the latter. 
Nelson and Oppen always allow relations of the form $x \neq y$ in the input, which we would like to avoid, because there are first-order 
expansions $\fA$ of $(\Q;<)$ with a polynomial-time tractable CSP 
%\red{where adding a relation symbol $\neq$ to $\tau_1$ which denotes $\{(x,y) \in \Q^2 \mid x \neq y\}$} 
where adding the relation $\neq$ to $\fA$   
makes the CSP hard, as the following examples shows. 

\begin{exa}
Let $\fA$ be the temporal structure $(\Q; <, R^{\min}_{\leq})$ where $R^{\min}_{\leq}$ is the relation defined by $\phi(x,y,z) \ceq x \geq y \OR x\geq z$.
% where $R^{\min} \ceq \{(a,b,c) \mid a\geq b \vee a \geq c\}$. 
Then
$\CSP(\fA)$ is in P by Theorem~\ref{thm:bod-kara} below because $\fA$ is preserved by $\min$. But $\CSP(\fA;\neq)$ is NP-hard by Theorem~\ref{thm:bod-kara} because $(\fA;\neq)$ is neither preserved by a constant operation, $\mi$, $\mx$, $\min$, nor by their duals. 
\end{exa}

An analysis of the correctness proof of the algorithm of Nelson and Oppen yields that the set of literals in the definition of convexity can be replaced by a set of atomic formulas if the input of the decision problem is restricted to a set of atomic formulas, i.e., we only require that \emph{$\neq$ is independent from $T_1$ and $T_2$} (see Definition~\ref{def:independence}). Independence of $\neq$, stably infinite theories, tractable CSPs and the presence of $\neq$ in the signature of $T_1$ and $T_2$ is what we refer to as the \emph{weakened} conditions of Nelson and Open.

Furthermore, Nelson and Oppen did not require that the signature is purely relational. However, this difference is rather a formal one, because a function can be represented by its graph and nested functions can be unnested in polynomial time by introducing new existentially quantified variables for nested terms. In Section~\ref{sec:ComplexityResults} we will prove a tractability criterion which is slightly stronger than the criterion of Nelson and Oppen with weakened conditions.

\subsection{Generic Combinations}\label{sec:genericCombinations}
In the context of combining decision procedures for CSPs, the notion of generic combinations has been introduced in~\cite{BodirskyGreinerCombinations}. However, others have studied such structures before (for instance in \cite{Oligo,Topo-Dynamics,42,LinmanPinsker}). 

\begin{defi}
Let $\fA_1$ and $\fA_2$  be countably infinite $\omega$-categorical structures with disjoint relational signatures $\tau_1$ and  $\tau_2$. 
A countable model $\fA$ of $\Th(\fA_1) \cup \Th(\fA_2)$
is called a \emph{generic combination of
$\fA_1$ and $\fA_2$} if for any $k\in \N$ and any $a, b \in A^k$ with pairwise distinct coordinates
\begin{align*}
& \Aut(\fA^{\tau_1})  a \cap \Aut(\fA^{\tau_2})  b \neq \emptyset  \quad\text{and}\\
& \Aut(\fA^{\tau_1})  a \cap \Aut(\fA^{\tau_2})  a  = \Aut(\fA)  a. 
\end{align*}
All generic combinations of $\fA_1$ and $\fA_2$ are  isomorphic (Lemma~2.8 in~\cite{BodirskyGreinerCombinations}), so we will speak of \emph{the} generic combination of two structures, and denote it by $\fA_1 * \fA_2$. 
\end{defi}

By definition, the $\tau_i$ reduct of $\fA \ceq \fA_1 * \fA_2$ is a model of $\Th(\fA_i)$, which is $\omega$-categorical, and therefore, $\fA^{\tau_i} \cong \fA_i$ for $i=1$ and $i=2$. 
Hence, we may assume without loss of generality that $\fA_1$, $\fA_2$, and $\fA$ have the same domain. 
It is an easy observation that an instance $\phi_1 \AND \phi_2$ of $\CSP(T_1 \cup T_2)$, where $\phi_i$ is a $\tau_i$-formula,  is satisfiable if and only if for $i=1$ and $i=2$ there exist models $\fA_i$ of $T_i$ with $\lvert A_1 \rvert = \lvert A_2 \rvert$ such that $\phi_i$ is satisfiable in $\fA_i$ and the satisfying assignments of $\phi_1$ and $\phi_2$ identify exactly the same variables. Therefore, the fact that $\CSP(\fA) = \CSP(\Th(\fA_1) \cup \Th(\fA_2))$ easily follows from $\Aut(\fA^{\tau_1})  a \cap \Aut(\fA^{\tau_2})  b \neq \emptyset $ and $\omega$-categoricity of $\fA_1$ and $\fA_2$.

A structure $\fA$ has \emph{no algebraicity} if every set defined by a first-order formula over $\fA$ with parameters from $A$ is either contained in the set of parameters or infinite. The following proposition characterises when generic combinations of $\omega$-categorical structures exist.

\begin{thmC}[Proposition~1.1 in~\cite{BodirskyGreinerCombinations}]\label{thm:existenceGenericComb}
Let $\fA_1$ and $\fA_2$ be countably infinite $\omega$-categorical structures with disjoint relational signatures. Then $\fA_1$ and $\fA_2$ have a generic combination if and only if either both $\fA_1$ and $\fA_2$ do not have algebraicity 
or one of $\fA_1$ and $\fA_2$ does have algebraicity and the other structure is preserved by all permutations. 
\end{thmC}

%%%%%%%%%%%%%%%%%%%%%%%%%%%%%%%%%%%%%%%%%%%%%%%%%%%%%%%%%%%%%%
 \subsection{Temporal Structures}\label{sec:TCLs}
 A relation with a first-order definition over $(\Q;<)$ is called \emph{temporal}. An example of a temporal relation is the relation $\Betw$ from the introduction.
 A \emph{temporal structure} is a relational structure $\fA$ with domain $\mathbb{Q}$ all of whose relations are temporal. The structure $(\mathbb{Q};<)$ is homogeneous, i.e., every order-preserving map between two finite subsets of $\mathbb{Q}$ can be extended to an automorphism of $(\Q;<)$. Therefore, the orbit of a tuple in $\fA$ is determined by identifications and the ordering among the coordinates. 
 It follows from Theorem~\ref{thm:ryll} that all temporal structures are $\omega$-categorical.

\subsubsection{Polymorphisms of Temporal Structures}\label{sec:polysOfTemporal}
One of the fundamental results in the proof of the complexity dichotomy for temporal CSPs, Theorem~\ref{thm:end} below, also plays an important role for combinations of temporal CSPs. To understand Theorem~\ref{thm:end} and for later use, we define the relations $\Cycl$, $\Betw$, and $\Sep$:
\begin{align*}%% DO NOT USE \set here!
 \Betw \ceq \{(x,y,z) \in \Q^3 \mid \,&(x < y < z) \OR (z < y < x)\}\\
 \Cycl \ceq \{(x,y,z) \in \Q^3 \mid \,&(x < y < z) \OR (y < z < x) \OR (z < x < y)\}\\
 \Sep \ceq \{(x,y,u,v)\in \Q^3 \mid \,&(x < u < y < v) \OR (y < u < x < v) \OR\\
    &(x < v < y < u) \OR (y < v < x < u)\}
\end{align*}

 \begin{thmC}[Bodirsky and K\'ara~\cite{tcsps-journal}, Theorem 20]
 \label{thm:end}
 Let $\fA$ be a temporal structure. Then at least one of the following cases applies.
 \begin{itemize}
 \item $\fA$ has a constant endomorphism; 
\item One of the relations $<$, $\Cycl$, $\Betw$, or $\Sep$ has a primitive positive definition in $\fA$. 
\item $\fA$ is preserved by all permutations of $\Q$. 
 \end{itemize}
 \end{thmC}

 We introduce several notions that are needed to describe the polynomial-time tractable temporal CSPs from~\cite{tcsps-journal}. 
However, as opposed to~\cite{tcsps-journal} we flip the roles of $0$ and $1$ in the following definition  
because in this way the resulting systems of equations are homogeneous (see Theorem~\ref{thm:SNF_base} (4) below;    we follow~\cite{RydvalFP}). 
 \begin{defi}
 For a tuple $t\in \mathbb{Q}^{n}$ we define the \emph{min-indicator function} $\chi\colon \mathbb{Q}^{n}\rightarrow \{0,1\}^{n}$ by $\chi(t)[i]\ceq 1$ if and only if $t[i]\leq t[j]$ for all $ 1\leq j \leq n$. The tuple $\chi(t)\in \{0,1\}^{n}$ is called the \emph{min-tuple} of $t\in \mathbb{Q}^{n}$. For an $n$-ary relation $R$ we define 
 $$\chi(R) \ceq \set{\chi(t) \mid t\in R} \text{ and }\chi_0(R) \ceq \chi(R) \cup \{(\underbrace{0,\dots,0}_{n \text{ zeros}})\}.$$
 \end{defi}
 Let $\min$ denote the binary minimum operation on $\mathbb{Q}$.
 For any fixed endomorphisms $\alpha,\beta,\gamma$ of $(\mathbb{Q};<)$ which satisfy $\alpha(a)<\beta(a)<\gamma(a)<\alpha(a+\epsilon)$ for every $a\in \mathbb{Q}$ and every $\epsilon\in \mathbb{Q}$ with $\epsilon >0$, the binary operation $\mi$ on $\mathbb{Q}$ is defined by
 \begin{displaymath}
   \mi(x,y)\ceq \left\{ \begin{array}{ll}  \alpha(x) & \text{ if } x= y, \\
                       \beta(y) & \text{ if } x>y, \\
                       \gamma(x) & \text{ if } x<y.
                     \end{array}  \right. 
 \end{displaymath}
The intuition behind this definition is best explained through illustrations; for such illustrations, additional explanation, and the argument why such functions do exist we refer the reader to~\cite{tcsps-journal} or~\cite{Book}; the same applies to the operations that are introduced in this section. 
 For %possibly different MB: irritierend, es ist doch gerade der Punkt dass es auf die alpha, beta nicht ankommt. 
  $\alpha, \beta$ satisfying the same conditions, $\mx$ is the binary operation on $\mathbb{Q}$ defined by
 \begin{displaymath}
  \mx(x,y)\ceq \left\{ \begin{array}{ll} 
    \alpha(\min(x,y)) & \text{ if } x\neq y, 	\\
    \beta(x) & \text{ if } x= y.  
    \end{array}   \right. 
 \end{displaymath}

 \begin{thmC}[\cite{RydvalFP}, Lemma~4.1 and Theorem~5.2] \label{thm:mixedmx} 
 We have 
 $$
\overline{\langle \{\mx\} \cup \Aut({\mathbb Q};<) \rangle} = \Pol({\mathbb Q};X)$$ where
 \[X \ceq \{(x,y,z)\in \mathbb{Q}^{3}\mid x=y<z\vee x=z<y\vee y=z<x\}.\]
Moreover, every temporal structure $\fB$ preserved by $\mx$ either admits a primitive positive definition of $X$ or is preserved by a constant operation or by $\min$.
\end{thmC}

Let $\ll$ be an arbitrary binary operation on $\mathbb{Q}$ such that $\ll(a,b)<\ll(a',b')$ if and only if 
 \begin{itemize}
 	\item $a\leq 0$ and $a<a'$, or
 	\item $a\leq 0$ and $a=a'$ and $b<b'$, or
 	\item $a,a'>0$ and $b<b'$, or
 	\item $a>0$ and $b=b'$ and $a<a'$.
\end{itemize}
Let $\lex \colon \Q^2 \rightarrow \Q$ be an arbitrary operation that induces the lexicographic order on $\Q^2$ (just like $\ll$ if the first argument is not positive). 
Let $\pp \colon \Q^2 \rightarrow \Q$ be an arbitrary operation such that $\pp(a,b) \leq  \pp(a',b')$ if and only if either
\begin{itemize}
        \item $a\leq 0$ and $a \leq a'$, or
        \item $0 < a$, $0<a'$ and $b \leq b'$ holds.
 \end{itemize}
Notice that the functions $\mi$, $\mx$, $\pp$, $\ll$, their duals, and $\lex$ are not uniquely specified by their definitions. They rather specify a unique weak linear order on $\Q^2$. By Observation~10.2.3 in~\cite{Bodirsky-HDR-v8}, any two functions in $\Pol(\Q;<)$ which generate the same weak linear order on $\Q^2$ are equivalent with respect to containment in subclones of $\Pol(\Q;<)$. Hence, we may assume the following additional properties for convenience:
\begin{itemize}
\item $\mx(0,0) = 1$ and $\mx(1,0) = 0$,
\item $\mi(0,0) = 0, \mi(1,0) = 1, \mi(0,1)=2, \mi(1,1)=3$,
\item $\ll(0, 0) = 0$, $\ll(1,0) = 1$, $\ll(2,0) = 2$, $\ll(3,0) = 3$ and $\ll(1,1) = 4$.
\end{itemize}
The polymorphisms we presented are connected by the following inclusions (see~\cite{tcsps-journal} or Chapter 12 in~\cite{Book}). For $m \in \set{\min, \mi, \mx}$ and $l \in  \set{\ll, \dll}$ we have 
\begin{align*}
 \overline{\langle \pp, \Aut(\Q) \rangle} & \subseteq  \overline{\langle m, \Aut(\Q) \rangle} ,\\
 \overline{\langle \dpp, \Aut(\Q) \rangle}
 &\subseteq \overline{\langle \textup{dual-}m, \Aut(\Q) \rangle} ,\\
\overline{\langle \lex, \Aut(\Q) \rangle} &\subseteq  \overline{\langle l, \Aut(\Q) \rangle}  .
\end{align*}

\subsubsection{Complexity of Temporal CSPs}
We can now state the complexity dichotomy for temporal CSPs. 

\begin{thmC}[\cite{tcsps-journal}, Theorem~50]\label{thm:bod-kara} Let $\fA$ be a temporal structure with finite signature.
Then one of the following applies:
 	  \begin{enumerate}
 	  	\item $\fA$ is preserved by $\min$, $\mi$, $\mx$, $\ll$, the dual of one of these operations, or by a constant operation. In this case $\CSP(\fA)$, is in P.\label{case:tractableCasesTCL}
 	  	\item $\CSP(\fA)$ is NP-complete.
 	  \end{enumerate}    
 \end{thmC} 
\noindent In our proofs, we also need some intermediate results from~\cite{tcsps-journal}.
In particular, we use the ternary temporal relation introduced in Definition~3 in~\cite{tcsps-journal}:
\begin{displaymath}
T_3 \ceq \set{(x,y,z) \mid x=y <z \vee x=z < y}
\end{displaymath}
$T_3$ is preserved by $\pp$, but by none of the polymorphisms listed in item (1) of Theorem~\ref{thm:bod-kara} and therefore 
$\CSP(\Q;T_3)$ is NP-complete.

\begin{thmC}[\cite{tcsps-journal}, Lemma~36]\label{thm:foundational2}
Let $\fA$ be a first-order expansion of $(\Q;<)$ preserved by $\pp$. 
Then either $T_3$ has a primitive positive definition in $\fA$, or $\fA$ is preserved by 
$\mi$, $\mx$, or $\min$.
\end{thmC}

\subsubsection{Known Syntactic Descriptions of Temporal Relations} 
We also need syntactic descriptions for temporal relations preserved by the operations introduced in the previous sections.

\begin{thmC}[\cite{ToTheMax} (Theorems~4, 5, and 6), \cite{Bodirsky-HDR-v8} (Proposition~10.4.7 and Theorem~10.5.18), and~\cite{tcsps-journal} (observation above Theorem~42)] \label{thm:SNF_base}
A temporal relation $R$ is preserved by
\begin{enumerate}
   \item \label{item:SNF_Pp}
$\pp$ if and only if $R$ can be defined by a conjunction of formulas of the form 
   \begin{align*}
   x_1 \circ_2 x_2 \vee \dots \vee x_1 \circ_n x_n
   \text{ where } \circ_i\in \{\neq,\geq\};
   \end{align*}
    \item \label{item:SNF_Min}
    $\min$ if and only if $R$ can be defined by a conjunction of formulas of the form
    \begin{align*}
    x_{1}\circ_{2}x_{2}  \vee \dots \vee x_{1} \circ_{n} x_{n} \text{ where } \circ_{i}\in \{>,\geq\}.  
    %\label{eq:SNF_Min}
    \end{align*}     
    \item    \label{item:SNF_Mi}
    $\mi$ if and only if $R$ can be defined  by a conjunction of formulas of the form
    \begin{align*}
    x_1 \circ_{2} x_2 \vee  \dots \vee x_1 \circ_n x_n  \text{ where } \circ_i \in \set{\neq, >, \geq} \text{ with at most one } \circ_i \text{ equal to } \geq. %\label{eq:SNF_Mi}
    \end{align*}
   \item \label{item:SNF_Mx}
$\mx$ if and only if $R$ can be defined by a conjunction of $\{<\}$-formulas 
   %\footnote{In \cite{TCLsFPpaper} such clauses are called \emph{Ord-Xor}.}
   $\phi(x_1,\dots, x_{n})$ 
   %where each $R_i$ is an $n_i$-ary temporal relation for which
   for which there exists a homogeneous system $A x=0$ of linear equations over $\GF_2$ 
   such that for every $t \in \mathbb{Q}^{n}$  
   \begin{align*}
   t \text{ satisfies } \phi \text{ if and only if }A \chi(t)=0. 
   \end{align*}
    In this case, there exists a homogeneous system $Ax=0$ of
linear equations over $\GF_2$  with solution space $\chi_0(R)$.
   \item \label{item:SNF_Ll} 
   $\ll$ if and only if $R$ can be defined by a conjunction of formulas of the form
   \begin{align*}
   (x_{1} >  x_{2} \vee \dots \vee x_{1} > x_{m})  \vee (x_{1}= \dots =x_{m}) \vee \bigvee_{m<2i< n} x_{2i} \neq x_{2i+1},
   \end{align*}
       where the clause $x_{1}= \dots =x_{m}$ may be omitted.
\end{enumerate}
\end{thmC}
\noindent Note that the relation $\J$ can equivalently be written as
\begin{displaymath}
 \J = \big \{(a,b,c) \in \mathbb{Q}^{3}\mid 
(a \geq b \vee a>c) \wedge (b \geq a \vee b>c) \}. 
\end{displaymath}
Theorem~\ref{thm:SNF_base} then implies that $\J$ is preserved by $\min$ and $\mi$.
To see that $\J$ is also preserved by $\mx$, note that $\chi_0(\J) = \{(1,1,1),(1,1,0),(0,0,1),(0,0,0)\}$, 
which is the solution space of the linear equation $x_1 = x_2$, and $\J$ contains all triples over ${\mathbb Q}$ whose $\min$-tuple satisfies $x_1 = x_2$. 
%$R^{\mix}$ is the
%$\J(x,y,z)$ is equivalent to 
%$$\exists h (X(z,z,h) \wedge X(x,y,h))$$
%and hence $\J$ has a primitive positive definition in $(\Q;X)$; since $X$ is preserved by $\mx$ (Theorem~\ref{thm:SNF_base}) the same is true for $\J$ (Theorem~\ref{thm:InvPol}).}  

Every temporal relation can be defined by a quantifier-free $\{<\}$-formula $\phi$
%(x_1, \dots, x_n)$ 
and one may assume that $\phi$ is written in \emph{conjunctive normal form (CNF)} 
$$\bigwedge_{\ell=1}^{k}  \bigvee_{i\in I_\ell}\phi_{\ell,i}$$ 
where $\phi_{\ell,i}$ is an atomic $\{<\}$-formula. 
 We say that $\phi$ is in \emph{reduced CNF} if we cannot remove any disjunct $\phi_{\ell,i}$ from $\phi$ without altering the defined relation. If $\phi$ is in reduced CNF, then for any $\ell\in [k]$ and $i \in I_\ell$ there exists $t\in R$ that satisfies 
 $\phi_{\ell,i}$ and does not satisfy any other 
disjunct $\phi_{\ell,j}$ for $j \in I_{\ell} \setminus \{i\}$. We use the symbols $\leq,\neq,\geq,>$ as the usual shortcuts, for $x < y \vee x = y$, etc. 
Clearly, every formula is equivalent to a formula
in reduced CNF. Remarkably, the syntactic form
in \ref{item:SNF_Min} is preserved by removing literals; hence, in \ref{item:SNF_Min} we may assume without loss of generality that the definition of $R$ is additionally reduced.

%%%%%%%%%%%%%%%%%%%%%%%%%%%%%%%%%%%%%%%%%%%%%%%%%%%%%

\subsection{Known Relational Generating Sets}
Many important temporal structures $\fA$ can also be described elegantly and concisely by specifying a finite set of temporal relations such that the temporal relations of $\fA$ are precisely those that have a primitive positive definition in $\fA$. Note that such a finite set might not exist even if $\fA$ contains all relations that are primitively positively definable in $\fA$. We need such a result for the temporal structure that contains all temporal relations preserved by $\pp$.

\begin{thmC}[\cite{Book}, Theorem 12.7.4]\label{thm:gen-pp}
A temporal relation is preserved by $\pp$ if and only if it has a primitive positive
definition in $(\Q;\neq,R^{\min}_{\leq},S^{\mi})$ where
\begin{align*}
R^{\min}_{\leq} & \ceq \set{(x,y,z) \in \Q^3 \mid x \geq y \OR x\geq z}\quad \text{ and }  \\
S^{\mi} & \ceq \set{(x,y,z) \in \Q^3 \mid x \neq y \vee x \geq z}. 
\end{align*}
\end{thmC}

\section{Polynomial-Time Tractable Combinations}\label{sec:ComplexityResults}
The following definition already appeared in~\cite{Book} and~\cite{BroxvallJonssonRenz} and is closely related to the convexity condition of Nelson and Oppen. The key difference to convexity of $T$ is that we consider  conjunctions of atomic formulas instead of conjunctions of literals.

\begin{defi}\label{def:independence}
 Let $T$ be a $\tau$-theory. We say that \emph{$\neq$ is independent from $T$} if for any conjunction of atomic $\tau$-formulas $\phi$ the formula $\phi \AND \bigwedge_{i=1}^k x_i \neq y_i$ is satisfiable in some model of $T$ whenever the formula $\phi \AND x_i \neq y_i$ is satisfiable in some model of $T$ for every $i\in [k]$. 
\end{defi}

The following is easy to see (see, e.g.,~\cite{Book}). 
\begin{prop}\label{thm:indep}
For every structure $\fA$ with a binary injective polymorphism, $\neq$ is independent from $\Th(\fA)$. 
\end{prop}

Nelson and Oppen require that both theories are stably infinite. We will make a weaker assumption captured by the following notion.

\begin{defi}
Let $T_1$ and $T_2$ be theories with signatures $\tau_1$ and $\tau_2$, respectively.
We say that $T_1$ and $T_2$ are \emph{cardinality compatible} if for all for $i\in [2]$ and all conjunctions $\phi_i(x_1,\dots,x_n)$ of atomic $\tau_i$-formulas, such that $\set{\exists x_1, \dots, x_n. \, \phi_i} \cup T_i$ has a model, there are models of $\set{\exists x_1, \dots, x_n. \, \phi_1} \cup T_1$ and $\set{\exists x_1, \dots, x_n. \, \phi_2} \cup T_2$ of equal cardinality. 
\end{defi}

Clearly, if $T_1$ and $T_2$ are stably infinite,  
then they are also cardinality compatible. 
Contrary to stably infinite theories where we require that we can choose the cardinality of the models to be countably infinite, the definition of cardinality compatibility also allows theories with finite models only. We also allow theories where some formulas are only satisfiable in finite models while others have infinite models,
as the following example shows. 

\begin{exa} 
Let $T$ be the theory $\set{\forall x,y \, ( \neg Q(x) \OR x = y)}$ whose signature only contains the unary relation symbol $Q$. There is an infinite model for $T$ where $Q$ is empty. However, if $\phi$ is the formula $Q(x)$, then all models for $T \cup \set{\phi}$ have exactly one element and this element is contained in $Q$. Hence, $T$ is not stably infinite, but cardinality compatible with itself. 
\end{exa}

The sufficient conditions for polynomial-time tractability of $\CSP(T_1 \cup T_2)$ given in the following theorem are slightly weaker than those by Nelson and Oppen.

\begin{thm}\label{thm:P}
Let $T_1$ and $T_2$ be cardinality compatible theories with finite, disjoint relational signatures and polynomial-time tractable CSPs. If $\neq$ is independent from both $T_1$ and $T_2$ and $\neq$ has an ep-definition in both $T_1$ and $T_2$, then $\CSP(T_1 \cup T_2)$ is polynomial-time tractable.
\end{thm}

\begin{proof}
  Let $\tau_1$ and $\tau_2$ be the signatures of $T_1$ and $T_2$, respectively. Let $S$ be a set of atomic $\tau_1 \cup \tau_2$-formulas with free variables among $x_1, \dots, x_n$. Then we may partition $S$ into $S_1$ and $S_2$ such that $S_i$ is a set of $\tau_i$-formulas and $S= S_1 \cup S_2$. Without loss of generality, we may assume that all variables occur in both $S_1$ and $S_2$ (this can also be attained by introduction of dummy constraints like $x= x$).
  Let $\phi_i(x,y)$ be an existential positive definition of $x\neq y$ in $T_i$ for $i\in\{1,2\}$.
  %, 
For $i=1$ and $i=2$ and for each tuple of variables $(x_k,x_l)$ and each disjunct $D( \cdot\,, \cdot)$ in $\phi_i$ we test whether $S_i \cup \set{D(x_k,x_l)}$ is satisfiable is some model of $T_i$. 
If, for a fixed tuple $(x_k,x_l)$, the answer is \emph{`unsatisfiable'} for all disjuncts of $\phi_i$, then we replace all occurrences of $x_l$ in $S_1$ and in $S_2$ by $x_k$. 
We iterate this procedure until no more replacements are made.
If $S_1$ or $S_2$ is unsatisfiable in all models of $T_1, T_2$ respectively thereafter, we return \emph{`unsatisfiable'}. Otherwise, we return \emph{`satisfiable'}.

To prove that this algorithm is correct, notice that if $S_i \cup \set{D(x_k,x_l)}$ is unsatisfiable for all disjuncts $D$ of $\phi_i$, then clearly $S_i \cup \set{x_k\neq x_l}$ is not satisfiable. Moreover, if $S_1$ or $S_2$ is unsatisfiable, then their union is unsatisfiable as well. Hence, the substitutions done by the algorithm do not change the satisfiability of $S_1\cup S_2$ in models of $T_1 \cup T_2$. Let us therefore assume that after the substitution process both $S_1$ and $S_2$ are satisfiable in some model of $T_1$ and $T_2$,  respectively. 
Without loss of generality we may assume that the variables $x_1, \dots, x_m$ remain in $S_1$ and in $S_2$. Furthermore, we know that $S_i \cup \set{x_k \neq x_l}$ is satisfiable for all $k \neq l$ with $k,l \leq m$ and both $i\in \set{1,2}$. 
Therefore, $S_i \cup \bigcup_{k\neq l} \set{x_k \neq x_l}$ is satisfiable in some model of $T_i$, because $\neq$ is independent from $T_i$, i.e., there exists $\fM_i \models T_i$ and an injective assignment $s_i \colon \set{x_1, \dots, x_m} \rightarrow \fM_i$ such that $\fM_i \models \bigwedge_{\sigma \in S_i}\sigma (s_i)$, 
where $\sigma(s_i)$ denotes $\sigma(s_i(y_1), \dots, s_i(y_k))$ and $y_1, \dots, y_k$ denote the variables in $\sigma$.
By the cardinality compatibility of $T_1$ and $T_2$, we may assume that $\fM_1$ and $\fM_2$ have the same cardinality. Therefore, there exists a bijection $f\colon M_1 \rightarrow M_2$ between their domains such that $f(s_1(x_k)) = s_2(x_k)$ for all $k\in [m]$. 
With this bijection we define a $\tau_1 \cup \tau_2$ structure $\fM$ which is a model of $T_1 \cup T_2$ via $R^\fM \coloneqq R^{\fM_1}$ for $R\in \tau_1$, and $a \in R^\fM$ if and only if $f(a) \in R^{\fM_2}$ for $R\in \tau_2$. This is well-defined, because the signatures of $T_1$ and $T_2$ are disjoint and because $s_1$ and $s_2$ are both injective. It is easy to verify that $\fM$ is a model of $T_1 \cup T_2$ and $\fM \models \bigwedge_{\sigma \in S_1} \sigma(s_1) \AND \bigwedge_{\sigma \in S_2} \sigma(s_1)$ and hence, the original instance is satisfiable.

The number of calls to the decision procedures for $T_1$ and $T_2$ is bounded by the number of pairs $(x_k, x_l)$ multiplied by the maximal number of rounds of substitutions and the number of disjuncts in $\phi_1 $ and $\phi_2$. Hence, the runtime of the algorithm is in $O(n^3)$.
\end{proof}

Notice that the tractability result by Nelson and Oppen can be obtained as a special case of Theorem~\ref{thm:P} when we consider theories which are stably infinite and where the set of atomic formulas is closed under negation. 
The following example shows that our condition covers strictly more cases already for combinations of temporal CSPs. 

\begin{exa}
For $i=1$ and $i=2$, let $({\mathbb Q};<_i,\leq_i)$ 
be a structure where $<_i$ denotes the usual strict linear order on the rational numbers, and $\leq_i$ denotes the corresponding weak linear order. 
Let $T_i \ceq \Th({\mathbb Q};<_i,\leq_i)$. 
Note that the relation $\neq$ does not have a primitive positive definition in $({\mathbb Q};<_i,\leq_i)$; however, it has the existential positive definition $x<_1 y \vee y <_1 x$. It is well-known that $\CSP({\mathbb Q};<_i,\leq_i)$ can be solved in polynomial time~\cite{PointAlgebra} and that $\neq$ is independent from $T_i$~\cite{BroxvallJonssonRenz}. 
Then $T_1$ and $T_2$ satisfy the conditions from Theorem~\ref{thm:P} but do not satisfy the conditions of Nelson and Oppen. 
\end{exa}
    
\section{The Operation \texorpdfstring{$\mix$}{mix}}
\label{sec:mix}
A certain temporal structure plays an important role in our proof; it contains the set of all temporal relations preserved by an operation, which we call $\mix$, and which is similar to the polymorphisms $\mi$ and $\mx$. We also present an equivalent description of these relations
in terms of syntactically restricted quantifier-free $\{<\}$-formulas (Theorem~\ref{thm:SNF_Mix}). 

\begin{defi}\label{def:mix}
Let $\alpha,\beta,\gamma$ be endomorphisms of $(\mathbb{Q};<)$ such that $\gamma(a)<\alpha(a)<\beta(a)<\gamma(a+\epsilon)$ for every $a, \epsilon\in \mathbb{Q}$ with $\epsilon >0$.
 Then $\mix$ is the binary operation on $\Q$ defined by
 \begin{displaymath}
   \mix(x,y)\ceq \left\{ \begin{array}{ll}  \alpha(x) & \text{ if } x < y, \\
                       \beta(x) & \text{ if } x=y, \\
                       \gamma(y) & \text{ if } x > y.
                     \end{array}  \right. 
 \end{displaymath}
 \end{defi}

In analogy to our convention in the case of the operations mi, mx, and ll, we fix some concrete values for the operation mix.
We claim that the endomorphisms $\alpha$, $\beta$, and $\gamma$ from the definition of mix can be chosen so that $\gamma(x)=3x$, $\alpha(x)=3x+1$, and $\beta(x)=3x+2$
for every $x\in \mathbb{Z}^{+}$. Figure~\ref{fig:mix} shows some values for $\mix$. 
For every $k\in \mathbb{Z}^+$, we define $\gamma_k,$ $\alpha_k$, and $\beta_k$ inductively as follows.
In the \emph{base case} $k=0$, we set $\alpha_0\coloneqq \delta_0 \circ \alpha$, $\beta_0\coloneqq \delta_0 \circ \beta$, and $\gamma_0\coloneqq \delta_0 \circ \gamma$, where $\alpha, \beta , \gamma$ are arbitrary operations satisfying the requirements in Definition~24 and $\delta_0$ is an automorphism of $(\mathbb{Q};<)$ such that 
$(\delta_0\circ \gamma)(0)=0 $, $(\delta_0\circ \alpha)(0)=1$, and $(\delta_0\circ \beta)(0)=2$.
Such $\delta_0$ exists because $(\mathbb{Q};<)$ is homogeneous and $\gamma(0)<\alpha(0)<\beta(0)$.
In the \emph{induction step} $k\to k+1$ we assume that, for every integer $0\leq \ell \leq k$, the endomorphisms  $\alpha_{\ell}$, $\beta_{\ell}$, and $\gamma_{\ell}$ of $(\mathbb{Q};<)$ are already defined and satisfy:
\begin{enumerate}
\item the requirements in Definition~24;
\item $\gamma_{\ell}(\ell)=3\ell$, $\alpha_{\ell}(\ell)=3\ell+1$, and $\beta_{\ell}(\ell)=3\ell+2$;
\item if $\ell>0$, then $\alpha_{\ell},$ $\beta_{\ell}$, and $\gamma_{\ell}$ take the same values as $\alpha_{\ell-1},$ $\beta_{\ell-1}$, and $\gamma_{\ell-1}$ on $(-\infty,\ell-1]$, respectively.
\end{enumerate}
We set $\alpha_{k+1}\coloneqq \delta_{k+1} \circ \alpha_k$, $\beta_{k+1}\coloneqq \delta_{k+1} \circ \beta_k$, and $\gamma_{k+1}\coloneqq \delta_{k+1} \circ \gamma_k$,
where $\delta_{k+1}$ is the identity map on $(-\infty,3k+2]$ and otherwise a piecewise affine transformation sending 
\begin{itemize}
\item $[3k+2, \gamma_k(k+1)]$ to $[3k+2,3(k+1)]$, 
\item $[\gamma_k(k+1), \alpha_k(k+1)]$ to $[3(k+1),3(k+1)+1]$, 
\item $[\alpha_k(k+1), \beta_k(k+1)]$ to $[3(k+1)+1,3(k+1)+2]$, and 
\item $[\beta_k(k+1), \infty)$ to $[3(k+1)+2,\infty)$.
\end{itemize}
Such $\delta_{k+1}$ is clearly an automorphism of $(\mathbb{Q};<)$ and $\alpha_{k+1}$, $\beta_{k+1}$, and $\gamma_{k+1}$ satisfy the items~1-3.\ from above.
The sequences $(\alpha_k)$, $(\beta_k)$, and $(\gamma_k)$ converge pointwise to endomorphisms $\alpha, \beta, \gamma$ of $(\mathbb{Q};<)$ with the desired properties.

% Analogously to the definitions of $\mi$ and $\mx$ we may without loss of generality fix some concrete values of $\alpha$, $\beta$, and $\gamma$. It is convenient to choose $\gamma \colon a \mapsto 3a, \alpha \colon a \mapsto 3a+1, \beta \colon a \mapsto 3a+2$ for $a\in \Z^+$. 

 \begin{figure}
\begin{center}
 \begin{tikzpicture}[scale=0.6]
 \draw[->,-latex] (-1.5,-1) -- (3.8,-1);  %bis 3
 \draw[->,-latex] (-1, -1.5) -- (-1,3.8);
 \foreach \i in {0,1,2,3} {
    \cxmarker{\i}{-1}{\i};    
    \cymarker{-1}{\i}{\i};
 }
 \node at (0,0) {2};
 \node at (1,0) {0};
 \node at (2,0) {0};
 \node at (3,0) {0};
 \node at (0,1) {1};
 \node at (0,2) {1};
 \node at (0,3) {1};
 \node at (1,1) {5};
 \node at (2,1) {3};
 \node at (3,1) {3}; 
 \node at (1,2) {4};
 \node at (1,3) {4};
 \node at (2,2) {8};
 \node at (3,2) {6};
 \node at (2,3) {7};
 \node at (3,3) {11};
\end{tikzpicture}
\end{center}
\caption{The image of $\mix$ on $\set{0,1,2,3}^2$.}
\label{fig:mix}
\end{figure}
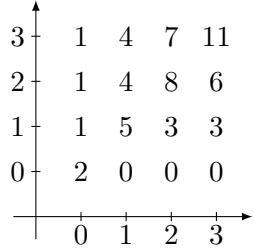

\begin{lem}\label{thm:mixGeneratesMi}
 The locally closed clone generated by $\mix$ and $\Aut(\Q;<)$ contains $\mi$.
\end{lem}
\begin{proof}
It is easy to check that $f(x,y) \ceq \mix(\mix(x,y), 3y)$ induces the same linear order as $\mi(x,y)$ on $(\Z^+)^2$. Hence, for any finite set $S \subseteq \Q$ there exist $\alpha, \beta, \gamma \in \Aut(\Q;<)$ such that $\alpha f(\beta(x), \gamma(y))|_{S^2} = \mi(x,y)|_{S^2}$. Then, by definition, $\mi \in \overline{\langle \mix, \Aut(\Q;<) \rangle}$.
\end{proof}

The relation $\J$ has the generalisation $\J_n$ of arity $n \geq 3$ defined as follows. 
 \begin{align}
   \J_n & \ceq \big \{(a_1,\dots,a_n) \in {\mathbb Q}^n \mid \min(a_3,\dots,a_n) \geq \min(a_1,a_2) \Rightarrow a_1=a_2 \big \} 
   \label{eq:Jn-def} 
 \end{align}
 Note that $\J_n(x_1,\dots,x_n)$ has the
 following definition in CNF 
$$\phi^{\mix}_n(x_1,\dots,x_n) \ceq \big ( x_1 \geq x_2 \OR \bigvee_{i\in \{3,\dots,n\}} x_1 > x_i ) \AND (x_2 \geq x_1 \OR \bigvee_{i\in \{3,\dots,n\}} x_2 > x_i \big \}$$
which is both of the form described in item~\ref{item:SNF_Min} and of the form described in~\ref{item:SNF_Mi} in Theorem~\ref{thm:SNF_base}. Hence, $\J_n$ is preserved by $\min$ and by $\mi$. 
Also note that $\J = \J_3$ and that $\J(a,b,c)$ is equivalent to $R^{\mi}(a,b,c) \wedge R^{\mi}(b,a,c)$ where 
  $$R^{\mi} \ceq \big \{(a,b,c)\in \mathbb{Q}^{3} \mid a \geq b \vee a > c \big \}.$$
The relation $\J_n$ is also preserved by $\mx$;
we first prove this for $\J_3$. 

%\begin{lem}\label{lem:XdefJ}
%The relation $\J$ is primitively positively definable in $({\mathbb Q};X)$ and hence is preserved by $\mx$. 
%.
%\end{lem}
%\begin{proof}
%It is easy to check that $\exists h \big(X(z,z,h) \AND X(x,y,h) \big)$  primitively positively defines $\J(x,y,z)$. The second part of the statement then follows from 
%Theorem~\ref{thm:mixedmx} and 
%Theorem~\ref{thm:InvPol}. 
%\end{proof}

\begin{lem}\label{thm:JgeneratesLongJ}
For every $n \geq 3$, the relation $\J_n$ 
%and $R^{\min}_n$ 
has a primitive positive definition in $(\Q;<,\J)$. 
\end{lem}
\begin{proof}
  A primitive positive definition of $\J_n$ can be obtained inductively by the observation that 
  $\J_n(x_1,\dots,x_n)$ is equivalent to the following formula.  
  \begin{align}
   \exists h \, \big (\J_{n-1}(x_1, h, x_3, \dotsc, x_{n-1}) \AND \J(h, x_2, x_n) \big) \label{eq:jn-induction}
  \end{align}
Every tuple $t \in \J_n$ satisfies (\ref{eq:jn-induction}): if $t$ satisfies $x_1 = x_2$ 
or if $t$ satisfies $x_n < \min(x_1,x_2)$, 
choose $h=x_1$; if $t$  satisfies $x_i < \min(x_1,x_2)$ for some $i \in \{3,\dots,n-1\}$, choose $h = x_2$. 
Conversely, suppose that $t \in {\mathbb Q}^n$ satisfies (\ref{eq:jn-induction}). If $t$ satisfies $x_1 = h$, then $t$ satisfies $x_1 = x_2 = h$ or $x_n < x_1 \AND x_n < x_2$ and therefore $\J_n$. The case that $t$  satisfies $x_2 = h$ is analogous. If $t$ satisfies $x_n < h \AND x_n < x_2$ and $x_i < x_1  \AND x_i < h$ for some $i \in \{3,\dots,n\}$, then it also satisfies $\min(x_i, x_n) < \min(x_1, x_2)$ and hence $t$ satisfies $\J_n$.
\end{proof}

\begin{lem}\label{lem:mix-jn}
For every $n \geq 3$, the operation $\mix$ preserves $\J_n$.
\end{lem}
 \begin{proof}
To prove that $\mix$ preserves $\J_n$ it suffices prove that $\mix$ preserves $\J$ due to Lemma~\ref{thm:JgeneratesLongJ} and Theorem~\ref{thm:InvPol}. 
%Let $(x_1, x_2, x_3)$ be names for the coordinates of $\J$. 
Suppose for contradiction that there are $t_1, t_2 \in \J$ such that $t_3 \ceq \mix(t_1, t_2) \not\in \J$. Then $t_3$ must satisfy $(x < y \AND x \leq z) \OR (y < x \AND y \leq z)$. 
Without loss of generality we may assume that $t_3$ satisfies the first disjunct. 
As $t_3[x]$ 
is minimal in $t_3$, the coordinate $x$ must be minimal in either $t_1$ or $t_2$. Assume the coordinate $x$ is minimal in $t_1$; the case with $t_2$ can be proven analogously. 
Then $t_1$ satisfies $x = y$ because $t_1 \in \J$. 
If $t_2$ satisfies $x = y$ then $t_3$ satisfies $x = y$, contrary to our assumptions. 
This implies that $t_2 \in \J$ satisfies $z < \min(x, y)$. If $t_2[z] < t_1[x]$ 
then $\min(t_1[z],t_2[z])=t_2[z] < \min(t_1[x],t_2[x])$, and hence $t_3[z] < t_3[x]$, a contradiction. Therefore, $\min(t_2[x], t_2[y]) > t_1[x]=t_1[y]$ and hence $t_3[x] = t_3[y]$, a contradiction. 
\end{proof}

\begin{thm}\label{thm:SNF_Mix}
  A temporal relation is preserved by $\mix$ if and only if it has a definition by a conjunction of clauses of the form 
  \begin{align}
    &\bigvee_{i=1}^n x\neq z_i \vee \bigvee_{i=1}^m x>y_i \qquad \text{for } n,m\in \N \label{eq:mixNf1}\\ 
   \text{and } \quad & 
   \phi^{\mix}_n(x_1,x_2,x_3,\dotsc, x_n) \qquad \text{for } n\geq 3.\label{eq:Jn}
  \end{align}
\end{thm}
\begin{proof}
  Let $R$ be a temporal relation preserved by $\mix$. Due to Lemma~\ref{thm:mixGeneratesMi}, the relation $R$ is also preserved by $\mi$. By Theorem~\ref{thm:SNF_base} case~\ref{item:SNF_Mi} the relation $R$ can be defined by a conjunction $\phi$ of clauses of the form
  \begin{align}
    & x \geq y \vee \bigvee_{i=1}^n x\neq z_i \OR \bigvee_{i=1}^m x > y_i  \qquad \text{for } n,m\in \N \label{eq:SNF_Mi}
  \end{align}
 where the literal $x\geq y$ can be omitted. Let $U_{\phi}$ be the set of clauses in $\phi$ which do have a literal of the form $x \geq y$ and which cannot be paired with another clause such that their conjunction is of the form $\phi^{\mix}_k$ for some $k$. Without loss of generality, we may assume that $\phi$ is chosen such that $\lvert U_{\phi}\rvert$ is minimal and such that no literal of the form $x \neq z_j$ can be replaced by $x > z_j$ without altering the relation defined by $\phi$. If $U_{\phi}$ is empty, then we are done. 
  Suppose towards a contradiction that $U_{\phi}$ contains a clause $C \ceq \big (x \geq y \vee \bigvee_1^n x\neq z_i \OR \bigvee_1^m x > y_i \big)$. Consider the new formulas $\phi_1, \dots, \phi_{n+3}$ obtained from $\phi$ by replacing $C$ by,
    respectively, 
  \begin{align}
    %& \bigvee_1^n x\neq z_i \OR \bigvee_1^m x > y_i \quad \text{(literal removal)},  \label{eq:replNogeq}\\
    &  x>y \vee \bigvee_1^n x\neq z_i \OR \bigvee_1^m x > y_i, % \quad \text{(literal replacement)}, 
     \label{eq:replReplgeq}\\
    &  x\geq y \vee x>z_1 \vee \bigvee_2^n x\neq z_i \OR \bigvee_1^m x > y_i, %\quad \text{(literal replacement)},  
    \label{eq:replReplneq}\\
    & \phi^{\mix}_{2+n+m}(x,y,z_1, \dots, z_n, y_1, \dots, y_m), \label{eq:replJxy} \\
   \text{or } \quad  &\phi^{\mix}_{2+n+m}(z_i, y, z_1, \dots, z_{i-1}, x, z_{i+1}, \dots,z_n, y_1, \dots, y_m) \quad \text{for some } i\in [n] . \label{eq:replJzy}   
  \end{align}
  Note that $\phi_j$ implies $\phi$ for each $j \in [n+3]$.
  Also note that if $\phi$ is equivalent to $\phi_j$ we found a contradiction to our choice of $\phi$ because either $\lvert U_{\phi_j} \rvert < \lvert U_{\phi} \rvert$ or we can replace a literal of the form $x\neq z_j$. This implies the existence of tuples $t_1, \dots, t_{n+3} \in R$ that do not satisfy $\phi_1, \dots, \phi_{n+3}$, respectively.
  We start the analysis of these tuples with the special case $n=0$. In this case we get
  \begin{itemize}
  \item a tuple $t_1 \in R$ that does not satisfy Clause~\eqref{eq:replReplgeq}. Since $t_1 \in R$ it must satisfy $U$, and hence it satisfies $x=y \AND \bigwedge_{i=1}^m x \leq y_i$; 
  \item a tuple $t_3 \in R$ that does not satisfy  Clause~\eqref{eq:replJxy}, i.e., $t_3$ satisfies $x>y \AND \bigwedge_{i=1}^m y \leq y_i$. 
  \end{itemize}
But then there exist $\alpha, \beta\in \Aut(\Q;<)$ such that  $t \ceq \mix(\alpha(t_3), \beta(t_1))$ does not satisfy $C$. The automorphisms $\alpha$ and $\beta$ have nothing to do with $\alpha$ and $\beta$ from
the definition of $\mix$, and their behaviour is illustrated in Table~\ref{tab:SNF_Mix1}.  The
automorphism $\alpha$ maps the coordinate of $t_3$ corresponding to $x$ in $C$ to
some value greater $0$. Likewise for the other entries of Table~1.

Therefore, $t$ does not satisfy $\phi$, contradicting the assumption that $R$ is preserved by $\mix$.
    \begin{table}
        \begin{center}%\label{tab:SNF_Mix1}
         $\begin{array}{r|rrc}
         & \multicolumn{1}{c}{x} & \multicolumn{1}{c}{y} & \multicolumn{1}{c}{y_{1 \leq i \leq m}} \\ \hline
      t_3' \ceq \alpha(t_3) & >0 & 0 & \geq 0  \\
      t_1' \ceq \beta(t_1) & 0 & 0 & \geq 0 \\ \hline
            \mix(t_3',t_1')& 0 & >0 & \geq 0
          \end{array}$
        \end{center}
        \caption{Calculation for the proof of Theorem~\ref{thm:SNF_Mix} in case $n=0$.}
        \label{tab:SNF_Mix1}
    \end{table}
If $n\geq 1$ the tuples are as follows:
\begin{itemize}
\item $t_1$ does not satisfy Clause~\eqref{eq:replReplgeq}, i.e., $t_1$  satisfies $x=y \AND \bigwedge_{i=1}^n x=z_i \AND \bigwedge_{i=1}^m x \leq y_i$;
\item $t_2$ does not satisfy Clause~\eqref{eq:replReplneq}, i.e., $t_2$  satisfies $x<y \AND x < z_1 \AND  \bigwedge_{i=2}^n x=z_i  \AND \bigwedge_{i=1}^m x \leq y_i$;
\item $t_4$ does not satisfy Clause~\eqref{eq:replJzy} for $i=1$, i.e., $t_4$  satisfies $$z_1 \neq y \AND (x\geq z_1 \OR x\geq y)\AND \bigwedge_{j=2}^n (z_j \geq z_1 \OR z_j \geq y) \AND \bigwedge_{j=1}^m (y_j \geq z_1 \OR y_j \geq y).$$ 
\end{itemize}
One of the following cases must apply: 
\begin{enumerate}
 \item $R$ contains $t_{4,z_1}$ satisfying $\psi_{z_1} \ceq y > z_1 \AND x > z_1 \AND  \bigwedge_{j=2}^n z_j \geq z_1 \AND \bigwedge_{j=1}^m y_j \geq z_1$;
 \item $R$ contains $t_{4,xz_1}$ satisfying $\psi_{xz_1} \ceq y > z_1 \AND x = z_1 \AND  \bigwedge_{j=2}^n z_j \geq z_1 \AND \bigwedge_{j=1}^m y_j \geq z_1$;
 \item $R$ contains $t_{4,y}$ satisfying $\psi_{y} \ceq z_1 > y \AND x > y \AND  \bigwedge_{j=2}^n z_j \geq y \AND \bigwedge_{j=1}^m y_j \geq y$;
 \item $R$ contains $t_{4,xy}$ satisfying $\psi_{xy} \ceq z_1 > y \AND x = y \AND  \bigwedge_{j=2}^n z_j \geq y \AND \bigwedge_{j=1}^m y_j \geq y$.
\end{enumerate}
Using suitable automorphisms $\alpha_1, \dots, \alpha_6 \in \Aut(\Q;<)$, we deduce the following (see Table~\ref{tab:SNF_Mix2}):
\begin{itemize}
\item in case (1)
there is also $t'''_{4,y} \in R$ satisfying $\psi_{y}$,
so we are also in case (3); 
\item in case (2) 
there is also $t''_{4,y} \in R$ satisfying $\psi_{y}$,
so we also in case (3); 
\item in case (3) 
the tuple $t^* \ceq \mix(t'_{4,y},t'_1) \in R$ 
does not satisfy $C$, a contradiction. 
\item in case (4) 
%If $t_{4,xy} \in R$ then 
there is also $t''_{4,z_1} \in R$ satisfying $\psi_{z_1}$, so we are also in case (3). 
\end{itemize}
Hence, in each case we reached a contradiction, which shows that the assumption that $U_\phi$ is non-empty must be false.
 \begin{table}
  \setlength{\myl}{7pt}
  \begin{center}%\label{tab:SNF_Mix2}
         $\begin{array}{r|rrrrc}
            &  \multicolumn{1}{c}{x} & \multicolumn{1}{c}{y} & \multicolumn{1}{c}{z_1} & \multicolumn{1}{c}{z_{2\leq i \leq n}} & \multicolumn{1}{c}{y_{1 \leq i \leq m}} \\ \hline
            t_2' \ceq \alpha_2(t_2) & 0 & >0   & >0  & 0\hspace{\myl} & \geq 0 \\
            t_1' \ceq \alpha_1(t_1) & 0 & 0   & 0 & 0\hspace{\myl} & \geq 0 \\ \hline
            t_{yz_1} \ceq \mix(t_2', t_1') & >0 & 0   & 0 & >0\hspace{\myl} & \geq 0 \\ \hline \hline        
            t'_{4,xz_1} \ceq \alpha_3(t_{4,xz_1})  & 0 & >0   & 0 & \geq 0\hspace{\myl} & \geq 0 \\                   t_{yz_1}                  & >0 & 0   & 0 & >0\hspace{\myl} & \geq 0 \\ \hline
            t_{4,y}'' \ceq \mix(t'_{4,xz_1}, t_{yz_1})  & >0 & 0   & >0 & >0\hspace{\myl} & \geq 0 \\            \hline \hline
            t_{4,xy}' \ceq \alpha_4(t_{4,xy}) & 0 & 0 & >0  & \geq 0\hspace{\myl} & \geq 0 \\
         t_1'                    & 0 & 0   & 0 & 0\hspace{\myl} & \geq 0 \\ \hline
            t_{4,z_1}'' \ceq \mix(t_{4,xy}',t_1')& >0 & >0 & 0  & \geq 0\hspace{\myl} & \geq 0 \\ \hline \hline
            t_{4,z_1}' \ceq \alpha_5(t_{4,z_1})& >0 & >0 & 0  & \geq 0\hspace{\myl} & \geq 0 \\ 
            t_{yz_1}                  & >0 & 0   & 0 & >0\hspace{\myl} & \geq 0 \\ \hline
            t_{4,y}''' \ceq \mix(t_{4,z_1}', t_{yz_1})  & >0 & 0   & >0 & >0\hspace{\myl} & \geq 0 \\ \hline \hline
            t_{4,y}' \ceq \alpha_6(t_{4,y})  & >0 & 0   & >0 & >0\hspace{\myl} & \geq 0 \\
            t_1'                   & 0 & 0   & 0 & 0\hspace{\myl} & \geq 0 \\ \hline
            t^* \ceq \mix(t_{4,y}', t_1') & 0 & >0   & 0 & 0\hspace{\myl} & \geq 0 \\
          \end{array}$
  \end{center}
  \caption{Calculation for the proof of Lemma~\ref{thm:SNF_Mix} in case $n\geq 1$.}
  \label{tab:SNF_Mix2}
 \end{table}

It remains to show that conjunctions of clauses of the form~\eqref{eq:mixNf1} and~\eqref{eq:Jn} are preserved by $\mix$. 
It suffices to verify that every relation defined by a single clause of this form is preserved by $\mix$. 
For the clauses of the form~\eqref{eq:Jn} we have already shown this in Lemma~\ref{lem:mix-jn}. 
Let $S$ be the relation defined by $\bigvee_{i=1}^n x\neq z_i \OR \bigvee_{i=1}^m x > y_i$. Suppose for contradiction that there exist $t_1, t_2 \in S$ such that $t_3 \ceq \mix(t_1, t_2) \not\in S$. Then $t_3$ must satisfy $x=z_1=\dots = z_n \AND \bigwedge_{i=1}^m x \leq y_i$. Therefore, either $t_1$ or $t_2$ must satisfy $C \ceq x=z_1=\dots = z_n > y_j$ for some $j$, because $\mix$ is only constant on a set of pairs if one coordinate is constant and the other coordinate is bigger or equal to the first one. Without loss of generality we may assume that $t_1$ satisfies $C$ with $j=1$.

If $t_2$ satisfies $C$ with some $j_2$ then $\min(t_1[x], t_2[x]) > \min(t_1[y_1], t_2[y_{j_2}])$ and therefore $t_3[x] > \min(t_3[y_1], t_3[y_{j_2}])$, contradicting $t_3 \not\in R$. If $t_2$ satisfies $x\neq z_j$ for some $j$, then $t_3[x] = t_3[z_j]$ implies that $\min(t_2[x], t_2[z_j]) > t_1[x]$. But then $t_3[x] > t_3[y_1]$, contradicting $t_3 \not\in S$.
\end{proof}

 \section{Primitive Positive Definability of the Relation \texorpdfstring{$\J$}{\Jtext}}\label{sec:DefinabilityOfJ}
 In this section we prove the following theorem.
 \begin{thm} \label{thm:dichotomy} Let $\fA$ be a first-order expansion of $({\mathbb Q};<)$ that is preserved by $\pp$. Then $\J$ has a primitive positive definition in $\fA$ if and only if $\fA$ is not preserved by $\ll$. 
 \end{thm}
 
The proof of this results is organised as follows. If the relation $T_3$ is primitively positively definable in $\fA$, then so is $\J$ (Proposition~\ref{thm:pp-new}). Otherwise, Theorem~\ref{thm:foundational2} implies 
that $\fA$ is preserved by $\mi$, $\mx$, or $\min$. 
It therefore suffices to treat first-order expansions $\fA$ of $({\mathbb Q};<)$ that are 
\begin{itemize}
\item preserved by $\mi$ (Section~\ref{sec:mi}),
\item preserved by $\mx$ but not by $\mi$ (Section~\ref{sec:mx}), and finally 
\item preserved by $\min$ but not by $\mi$ and not by $\mx$ (Section~\ref{sec:min}). 
\end{itemize}

 \subsection{Temporal Structures Preserved by \texorpdfstring{$\mi$}{mi}}\label{sec:mi}
In this section we prove Theorem~\ref{thm:dichotomy} for first-order expansions $\fA$ of $({\mathbb Q};<)$ that are preserved by $\mi$ 
(Proposition~\ref{thm:mi}). 
For this purpose, it turns out to be highly useful to distinguish whether the relation $\leq$ has a primitive positive definition in $\fA$ or not. 
If yes, then the statement can be shown directly 
%we may find a primitive positive definition of $\J$ relatively directly 
(Proposition~\ref{thm:mi+leq}).
Otherwise, $\fA$ is preserved by the operation $\mix$ from Section~\ref{sec:mix}  (Proposition~\ref{thm:produce-mix}). Then the syntactic normal form for temporal relations preserved by $\mix$ from Section~\ref{sec:mix} can be used to show the statement. 

\begin{prop}\label{thm:mi+leq}
Let $\fA$ be a first-order expansion of $({\mathbb Q};\leq)$ which is preserved by $\mi$ but not by $\ll$. Then $R^{\mi}$ 
and $\J$ have a primitive positive definition in $\fA$.
\end{prop}
\begin{proof} Let $R$ be a relation of $\fA$ which is not preserved by $\ll$. As $R$ is preserved by $\mi$, Theorem~\ref{thm:SNF_base}~\ref{item:SNF_Mi} implies that $R$ can be defined by a conjunction $\phi$ of clauses of the form
  \begin{displaymath}
     x \geq y \vee  \bigvee_{i=1}^m x > y_i \vee \bigvee_{i=1}^n x\neq z_i .
  \end{displaymath}
  We may assume that the literals
  $x > y_1,\dots,x > y_{m}$ cannot be removed from 
  such clauses without changing the relation defined by the formula. 
  As $R$ is not preserved by $\ll$,  Theorem~\ref{thm:SNF_base}~\ref{item:SNF_Ll} implies that $\phi$
 must contain a conjunct $C$ of the form $x \geq y \vee  \bigvee_{i=1}^m x > y_i \vee \bigvee_{i=1}^n x\neq z_i$ where $m \geq 1$. 
 Assume for contradiction that $\phi \wedge x=y$ implies $x=y_1=\dots =y_m \OR \bigvee_{i=1}^m x>y_i \OR \bigvee_{i=1}^n x \neq z_i$. 
 Then we can replace $C$ by $x > y \OR \bigvee_{i=1}^m x>y_i \OR x=y=y_1=\dots=y_m \OR \bigvee_{i=1}^n x \neq z_i$. However, if this is possible for all $C$ with $m\geq 1$, then $R$ is preserved by $\ll$, contradiction. So we may suppose that there exists a tuple $t_1 \in R$ and $j \in [m]$ such that
 \begin{align*}
    t_1  \text{ satisfies }\quad & x=y \AND x<y_j \AND \bigwedge_{i \neq j} x \leq y_i \AND \bigwedge_{i=1}^n x=z_i.
 \end{align*}  
  For the sake of notation, we assume that $j=1$. 
  As the literal $x>y_1$ 
  can not be removed from $C$ without changing the relation defined by $\phi$, 
   there is a tuple $t_2 \in R$ such that
  \begin{align*}
    t_2  \text{ satisfies }\quad & y > x \AND x > y_1 \AND  \bigwedge_{i\neq j} x \leq y_i  \AND \bigwedge_{i=1}^n x=z_i.    
  \end{align*}
  We may assume that $x,y,y_1, \dots, y_m, z_1, \dots, z_n$ refer to the first $2+m+n$ coordinates of $R$, in that order. Choose $k\in \N$ such that  $2+m+n+k$ is the arity of $R$ and let $u_1, \dots, u_k,y',z$ be fresh variables. The following is a primitive positive definition of $R^{\mi}$ in $\fA$:
  \begin{align*} 
    \psi(x,y',z) %y,y_1) 
    \ceq \; \exists & y,y_1,y_2,\dots, y_m, z_1, \dots, z_n, u_1, \dots, u_{k} \Big( y' \leq y \wedge z \leq y_1 \\
                  & \wedge  \, R(x,y,y_1,\dots, y_m, z_1, \dots, z_n, u_1, \dots, u_{k})
                   \AND \bigwedge_{i=2}^m x \leq y_i  \AND \bigwedge_{i=1}^n x=z_i\Big)
  \end{align*}
  %u_{\ar(R)-m-n-2}
 To see this, first note that the quantifier-free part of $\psi$
 implies that $x \geq y \vee x > y_1$,
 and hence that $x \geq y' \vee x > z$. 
 
 Conversely, choose $(a,b,c)\in R^{\mi}$. If $a \geq b$ then choose $\alpha \in \Aut(\Q;<)$ such that $\alpha(t_1[x])=a$ and $\alpha(t_1[y_1]) \geq c$ and set $y'=b$ and $z=c$. This is possible because $t_1[y_1] > t_1[x]$.
 Then $\alpha(t_1)$ provides values for $y,y_1, \dots, u_k$ which satisfy  
 all conjuncts of $\psi$: the conjunct $R(x,y,y_1,\dots)$ is satisfied because $\alpha(t_1)\in R$, and for the other conjuncts this is immediate. Hence, $\psi(a,b,c)$ holds.
  If $a > c$ then choose $\alpha \in \Aut(\Q;<)$ such that $\alpha(t_2[x]) = a$, $\alpha(t_2(y)) \geq b$ and $\alpha(t_2[y_1]) = c$, $y'=b$ and $z=c$. This is possible because $t_2[y] > t_2[x] > t_2[y_1]$.
Then $\alpha(t_2)$ provides values for $y,y_1, \dots, u_k$ which satisfy  
 all conjuncts of $\psi$:
the conjunct $R(x,y,y_1,\dots)$ is satisfied because $\alpha(t_2)\in R$
and for the other conjuncts this is immediate. 
\end{proof}

\begin{lem}\label{thm:mixClosureFromLeq}
%  Let $g$ be a binary operation that preserves $<$  and does not preserve $\leq$. Then the clone generated by $\Aut(\Q;<) \cup \{\mi, g\}$ contains a binary operation $f$ 
Let $\fA$ be a first-order expansion of $(\Q;<)$ 
which is preserved by $\mi$ and where $\leq$ is not primitively positively definable. Then
$\fA$ has a binary polymorphism $f$ 
such that for all positive $a_1, a_2,b_1,b_2\in \Q$
  \begin{align}   
    &2 = f(0,0) > f(0, b_1) = 1 = f(0, b_2) > f(a_1,0) = f(a_2,0) = 0.  \label{eq:mixClosure}
  \end{align}
\end{lem}
\begin{proof}
By Theorem~\ref{thm:InvPol} there exists
a polymorphism of $\fA$ that does not preserve $\leq$. There is also a binary polymorphism $g$ with this property, by Lemma~10 in~\cite{tcsps-journal}. 
  We can without loss of generality assume that there exist $p_1, p_2,q\in \Q$ such that $p_1 < p_2$ and $g(p_1, q) > g(p_2,q)$. Define $g' \ceq \gamma g(\alpha, \beta)$ with $\alpha,\beta,\gamma\in\Aut(\Q;<)$ such that $\alpha^{-1}(p_1, p_2) = (0,1)$, $\beta^{-1}(q) = 0$, and $\gamma(g(p_1, q), g(p_2,q)) = (1,0)$. Then $g'(0,0)= 1$ and $g'(1,0) = 0$. Defining $g''(x,y) \ceq  g'( \mi(x,y), y)$ we get 
  $g''(0,0) = g'(0,0) = 1$ and for all $c>0$
  we get
  %\begin{align*}
  $g''(c, 0) = g'(1,0) = 0$
  % \text{ and } 
  and $g''(0,c) = g'(2,c) =: d > 1$.
  %\end{align*}
  Defining $f(x,y) \ceq \mi(g''(y,x), g''(x,y))$ we get $f(0,0) = \mi(1,1) = 3$, and for all $c>0$ we get 
  $f(c,0)= \mi(d,0) = 1$, and $f(0,c) = \mi(0,d) = 2$.
  As $x\mapsto x-1$ is in $\Aut(\fA)$, the function $(x,y) \mapsto f(x,y)-1$ satisfies~\eqref{eq:mixClosure}.
\end{proof}

The following proposition  is similar to Proposition~10.5.13 in~\cite{Bodirsky-HDR-v8}.

\begin{prop}\label{thm:produce-mix}
Let $\fA$ be a temporal structure preserved by $\pp$ such that $\leq$  does not have a primitive positive definition in $\fA$. 
Then $\fA$ is preserved by $\mix$. 
\end{prop}

\begin{proof}
Let $R$ be a $k$-ary relation of $\fA$ and $r,s\in R$. We have to show that $t \ceq \mix(r,s)$ is in $R$. Let $\alpha,\beta,\gamma \in \End(\Q;<)$ be from the definition of $\mix$. Let $v_1 < \dots < v_l$ be the shortest sequence of rational numbers such that $t_i \in \bigcup_{j\in [l]} \set{\alpha(v_j), \beta(v_j), \gamma(v_j)}$ for every $i\in [k]$. For every $j\in [l]$ we define
\begin{displaymath}
 M_j \ceq  \big \{i\in [k] \mid t_i \in \set{\alpha(v_j), \beta(v_j), \gamma(v_j)} \big\}.
\end{displaymath}
Observe that $M_1, \dots, M_l$ is a partition of $[k]$ and therefore defines a partition on $\{t_1, \dotsc, t_k\}$. Furthermore, for each $i\in M_j$ either $v_j = r_i \leq s_i$ or $v_j = s_i \leq r_i$ holds. This defines a partition of $M_j$ into three parts:
\begin{align*}  \label{eq:3}
  M_j^{\alpha} & \ceq \{i\in M_j \mid v_j = r_j < s_j\},\\
  M_j^{\beta} & \ceq\{i \in M_j \mid v_j = r_j = s_j \}, \\
\text{and } \quad  M_j^{\gamma} & \ceq \{i \in M_j \mid v_j = s_j < r_j\}.
\end{align*}
  Let $\alpha_1,\dots, \alpha_l \in \Aut(\Q;<)$ be such that $\alpha_j(v_j) = 0$ for all $j \in [l]$. 
  By Lemma~\ref{thm:mixClosureFromLeq} there is a binary $f \in \Pol(\fA)$ 
  satisfying~\eqref{eq:mixClosure}. 
  For each $j \in [l]$ we define
  \begin{displaymath}
    u^j \ceq \pp \big (f(\alpha_j r, \alpha_j s), \pp(\alpha_j s, \alpha_j r) \big) %%%
  \end{displaymath}
  It is easy to verify that for all $i \in M_j$ and $w,w' >0$ 
  \begin{align*}    \label{eq:2}
    \text{if } i \in M^{\alpha}_j \text{ then } & u^j_i = \pp(f(0,w), \pp(w', 0)) = \pp(1, \pp(1,0)),\\
    \text{if } i \in M^{\beta}_j \text{ then }  & u^j_i = \pp(f(0,0), \pp(0,0)) = \pp(2,\pp(0,0)), \\
\text{and if } i \in M^{\gamma}_j \text{ then }  & u^j_i = \pp(f(w, 0), \pp(0,w')) = \pp(0,0).
  \end{align*}
In particular, $u^j$ is constant on each of $M_j^{\alpha}$, $M_j^{\beta}$, $M_j^{\gamma}$   and $u^j_i > u^j_{i'}$ for $i \in M^{\alpha}_j$ and $i' \in M^{\beta}_j$. 
We apply $f$ again to obtain 
$z^j \ceq f(\alpha_j r, \beta_j u^j)$ 
where $\beta_j \in \Aut(\Q;<)$ is such that 
$\beta_j(\pp(2, \pp(0,0))) = 0$. 
Then we get
for all $i \in M_j$ and $w > 0$  that 
\begin{align*}
   \text{if } i \in M^{\alpha}_j \text{ then } & z^j_i = f(0,w) = 1, \\
   \text{if } i \in M^{\beta}_j \text{ then } & z^j_i = f(0,0) = 2, \\
   \text{and if } i \in M^{\gamma}_j \text{ then } & z^j_i = f(w,e) < f(0, e') =0 \text{ for some } e' < e < 0.
\end{align*}
Thus, we found $z^1, \dots, z^l \in R$ such that 
for all $i \in M_j^{\beta}$, $i' \in M_j^{\alpha}$, and 
$i'' \in M_j^{\gamma}$ we have 
$z^j_i > z^j_{i'} > z^j_{i''}$.
Take any $j,j'\in [l]$ such that $j < j'$ and choose $i \in M_j^{\beta}$ and $i' \in M_{j'}$. Then $v_j = r_i = s_i < v_{j'} = \min(s_{i'}, r_{i'})$ and therefore $z^j_i < z^{j'}_{i'}$ because $f$, $\pp$, and all automorphisms preserve $<$. 
Therefore, we can apply Lemma~10.5.3 in~\cite{Bodirsky-HDR-v8} to $z^1, \dots, z^l$ which yields the existence of a tuple $t^* \in R$ with satisfies
$t^*_i < t^*_{i'}$ if and only if there exists $j<j'$ such that $i \in M_j, i' \in M_{j'}$, and $z^j_i < z^{j'}_{i'}$.
However, this is the same ordering that $t$ satisfies and hence, $t\in R$.
\end{proof}

\begin{prop}\label{thm:mix}
Let $\fA$ be a first-order expansion of $({\mathbb Q};<)$ preserved by 
$\mix$ but not by $\ll$. Then
% either $R$ is preserved by $\ll$ or 
$\J$ has a primitive positive definition in $\fA$.
\end{prop}

\begin{proof}
Let $R$ be a relation in $\fA$ that is not preserved by $\ll$. 
Lemma~\ref{thm:SNF_Mix} implies that
 $R$ can be defined by conjunctions of clauses the form~\eqref{eq:mixNf1} and~\eqref{eq:Jn}. As $R$ is not preserved by $\ll$, any such definition  must include at least on clause of the form~\eqref{eq:Jn}. Consider a clause of the form~\eqref{eq:Jn}, written in CNF $\phi^{\mix}_n = C_x \AND C_y$ with
    \begin{align*}
      C_x &\ceq \left(x \geq y \OR \bigvee_{i=1}^n x > z_i \right) & 
      C_y &\ceq \left(y \geq x \OR \bigvee_{i=1}^n y > z_i \right)  .
    \end{align*}
{\bf Claim 1.} Suppose that the literal $x \geq y$ can be replaced by $x > y$ in $C_x$ without changing the relation defined by $\phi$.
Then we can also replace the literal $y \geq x$ by $y > x$ in $C_y$ without changing the relation defined by $\phi$. 

The assumption implies that if $x\geq y$ is satisfied by a tuple $t\in R$ then either $t$ satisfies $x > y$, or $t$ satisfies $x=y$ and there exists $i$ such that $t$ satisfies $x>z_i$. In the first case $t$ satisfies $y>z_j$ (in order to satisfy $C_y$) and hence $t$ still satisfies $\phi$ after replacing $y\geq x$ by $y>x$ in $C_y$.
 In the second case,  % 
  $t$ satisfies $y=x>z_i$ and thus again satisfies $C_y$ after the same replacement. 

{\bf Claim 2.} 
Suppose that for some $i \in [n]$, the literal $x>z_i$ can be removed from $C_x$ without changing the relation defined by $\phi$. Then $y > z_i$ can be removed from $C_y$ without changing the relation defined by $\phi$.

  Case 1: All tuples $t\in R$ satisfy $x \leq z_i$, i.e., $x>z_i$ is never true.
    If there is $t\in R$ such that $t$ satisfies $y>z_i$, then $t$ also satisfies $y>x$. Hence, we can also remove %$x>z_i$ and 
    $y>z_i$ from $C_y$ without altering the relation defined by the formula.
    
    \begin{table}
    \begin{center}
         $\begin{array}{r|rrrr}
         & \multicolumn{1}{c}{x} & \multicolumn{1}{c}{y} & \multicolumn{1}{c}{z_i} & \multicolumn{1}{c}{z_{j\neq i}}\\ \hline
        t'\ceq \alpha_1(t_{y,i}) & 2 & 1 & 0 & \geq 1  \\
        t_c' \ceq \alpha_2(t_c) & 1 & 1 & \geq 1 & \geq 1\\ \hline
  \mix(t',t_c')& 3 & 5 & 1 & \geq 3 
          \end{array}$
    \end{center}  
    \caption{Calculation for Claim 2 (Case~2) in the proof of Proposition~\ref{thm:mix}.}
        \label{tab:MixJnPruning}
    \end{table}
    
   Case 2: There exists $t \in R$ where $x > z_i$ holds. 
    Suppose for contradiction that there exists a tuple $t_{y,i}\in R$ which does not satisfy $C_y$ after deletion of $y>z_i$ in $C_y$. Then
    $$t_{y,i} \text{ satisfies } x > y \AND y > z_i \AND \bigwedge_{j\neq i} z_j \geq y.$$
    As we already know that literal replacement can be applied to $C$ (Claim 1), we can assume that no literal in $\phi$ can be replaced. Therefore, there exists $t_c\in R$ such that
    \begin{displaymath}
      t_c \text{ satisfies }\quad x=y \wedge \bigwedge_{i=1}^n x\leq z_i.
    \end{displaymath}
    Then there exist $\alpha_1, \alpha_2\in \Aut(\Q;<)$ such that $\mix(t_{y,i}, t_c)$ satisfies $y>x \AND x>z_i \AND \bigwedge_{j\neq i} z_j\geq x$ (see Table~\ref{tab:MixJnPruning}), contradicting the assumption that we can remove $x>z_i$.

Claims 1 and 2 imply that we may assume without loss of generality that the literal $x \geq y$ cannot be replaced by $x > y$,
 that the literal $x > z_i$ cannot be removed from $C_x$ and, symmetrically, that $y > z_i$ cannot be removed from $C_y$ without changing the relation defined by $\phi$. 
Hence, there are $t_c, t_{x,i}, t_{y,i} \in R$ such that for all $1 \leq i \leq n$
\begin{align*}
    t_c \text{ satisfies }\quad & x=y \wedge \bigwedge_{i=1}^n x\leq z_i,  &   t_{x,i} \text{ satisfies }\quad & z_i < x < y \AND \bigwedge_{j\neq i} x\leq z_j, \\
 && \quad \text{ and $t_{y,i}$  satisfies }\quad & z_i < y < x \AND \bigwedge_{j\neq i} y\leq z_j.
\end{align*}
Now we apply automorphisms and $\mix$ to $t_c$, $t_{x,i}$, and $t_{y,i}$ to prove that $R$ contains tuples with more specific properties. 
We first prove that $R$ must contain a tuple $t_{c}^*$ satisfying 
\begin{align}
x=y \AND \bigwedge_{i \in [n]} x< z_i.
\label{eq:goal}
\end{align}
Choose $t_c$ as above such that the number $m$ of indices $j \in [n]$ such that $t_c$ satisfies $x<z_j$ is maximal. If $m = n$, then $t_c$ satisfies~\eqref{eq:goal} and hence satisfies the requirements for $t_c^*$. 
Otherwise, there exists $i \in [n]$ such that
$t_c$ satisfies $x = z_i$; this case will lead to a contradiction. 
%\item $t_c$ satisfies $x=y=z_1 = \cdots = x_n$. 
%\item there are $i,j,k \in [n]$ such that 
%$t_c$ satisfies $z_j > z_i = z_k = x$. 
%\end{itemize}
Choose  automorphisms $\alpha_1, \alpha_2 \in \Aut(\Q;<)$ such that $\alpha_1(t_c)$ satisfies $x=0$ 
and $\alpha_2(t_{x,i})$ satisfies $z_i = 0$. 
Then $$t_c' \ceq \mix(\alpha_1(t_c),\alpha_2(t_{x,i})) \in R$$ satisfies 
$z_i > x$ and $x=y$ (see Table~\ref{tab:mixJ1}). Moreover, if $k \in [n]$ is such that $\alpha_1(t_c)$ satisfies $x < z_k$ then $t_c'$ satisfies $x < z_k$ as well. Hence, the number $m$ of indices $j \in [n]$ such that $t_c'$ satisfies $x < z_j$ is at least  $m+1$, a contradiction to the choice of $t_c$.

\begin{table}
  \begin{center}
         $\begin{array}{r|rrrr}
         & \multicolumn{1}{c}{x} & \multicolumn{1}{c}{y} & \multicolumn{1}{c}{
         z_k} & \multicolumn{1}{c}{z_i} \\ \hline
        \alpha_1(t_c) & 0 & 0 & >0 & 0 \\
        \alpha_2(t_{x,i}) & 1 & 2 & \geq 1 & 0 \\ \hline
 t_{c}' \ceq \mix(\alpha_1(t_c),\alpha_2(t_{x,i})) & 1 & 1 & >2 & 2 
          \end{array}$
        \end{center}  
        \caption{Calculation of $t_{c}'$ in the proof of Proposition~\ref{thm:mix}.}
        \label{tab:mixJ1}
    \end{table}
         
         Our next goal is to prove the existence of $t_{x,i}^*, t_{y,i}^* \in R$ such that 
\begin{align*}
        t_{x,i}^*  \text{ satisfies }\quad & z_i < x < y \AND \bigwedge_{j\neq i} y<z_j \\
       \text{and }  \quad t_{y,i}^* \text{ satisfies }\quad & z_i < y < x \AND \bigwedge_{j\neq i} x<z_j.
\end{align*}
  Using $t_{c}^*$ and appropriately chosen $\alpha_1, \dots, \alpha_5 \in \Aut(\Q;<)$ we may first produce $h_{x,i},h_{y,i} \in R$ and combine them to get $t_{x,i}^*, t_{y,i}^* \in R$ as shown in Table~\ref{tab:mixJ2}.
  \begin{table}
   \begin{center}
         $\begin{array}{r|rrrr}
         & \multicolumn{1}{c}{x} & \multicolumn{1}{c}{y} & \multicolumn{1}{c}{z_i} & \multicolumn{1}{c}{z_{j\neq i}}\\ \hline
        t_{x,i}'\ceq \alpha_1(t_{x,i}) & 1 & 2 & 0 & \geq 1  \\
        {t_c^*}' \ceq \alpha_2(t_{c}^*) & 1 & 1 & >1 & >1 \\ \hline
 h_{x,i} \ceq \mix(t_{x,i}',{t_{c}^{*}}')& 5 & 3 & 1 & \geq 4 \\ \hline \hline
        t_{y,i}'\ceq \alpha_3(t_{y,i}) & 2 & 1 & 0 & \geq 1  \\
        {t_c^*}' & 1 & 1 & >1 & >1 \\ \hline
 h_{y,i} \ceq \mix(t_{y,i}',{t_{c}^{*}}')& 3 & 5 & 1 & \geq 4\\ \hline \hline
        h_{x,i}'\ceq \alpha_4(h_{x,i}) & 3 & 1 & 0 & \geq 2  \\
        h_{y,i}' \ceq \alpha_5(h_{y,i}) & 1 & 3 & 0 & \geq 2 \\ \hline
 t_{x,i}^* \ceq \mix(h'_{x,i},h'_{y,i})& 3 & 4 & 2 & \geq 6 \\ \hline \hline
        h_{y,i}' & 1 & 3 & 0 & \geq 2  \\
        h_{x,i}' & 3 & 1 & 0 & \geq 2 \\ \hline
 t_{y,i}^* \ceq \mix(h_{y,i}',h_{x,i}')& 4 & 3 & 2 & \geq 6 \\
          \end{array}$
\end{center}
       \caption{Calculation of $t_{x,i}^*$ and $t_{y,i}^*$ in the proof of Proposition~\ref{thm:mix}.}
        \label{tab:mixJ2}
\end{table}         

Without loss of generality we may assume that $x,y,z_1, \dots, z_n$ correspond to the first $n+2$ coordinates in $R$. Let $u_1, \dots, u_m$ be fresh variables such that the arity of $R$ is $2+n+m$
and define 
%Because $t_x^*, t_y^*, t_c^* \in R$, the following is a primitive positive definition of $\J$:
\begin{align*}
  \psi(x,y,\bar z,\bar u) & \ceq  R(x,y,\bar z, \bar u)  \wedge   \bigwedge_{i=2}^n x<z_i \AND y < z_i \\
 \text{ and } \quad  \psi'(x,y,z) & \ceq \exists z_1, \dots, z_k,u_1, \dots, u_m \big( \psi(x,y,\bar z,\bar u)  \AND z<z_1 \big). 
\end{align*}
To show that $\psi'$ defines $\J$, first notice that 
$t_{x,1}^*$, $t_{y,1}^*$, and  $t_c^*$ satisfy $\psi$ and that $\psi$ implies
$x \geq y \vee x > z_1$ and $y \geq x \vee y > z_1$ 
because all disjuncts of $C_x$ and $C_y$ involving $z_2, \dotsc, z_n$ do not hold. 
This in turn implies that 
the set of orbits of $(x,y,z_1)$ in tuples that satisfy $\psi$ is contained in $\J$. 
It follows that if $(a, b, c)$ satisfies $\psi'$, 
then either $a=b$, or there exists $z_1$ such that 
$c < z_1 < \min(a,b)$, so $(a,b,c) \in \J$. 

Conversely, let $(a,b,c)$ be in $\J$. If $a=b$ and we may choose $\alpha\in \Aut(\fA)$ such that $\alpha(t_c^*[x]) = a$ and $\alpha(t_c^*[z_1]) > c$, in which case $\alpha(t_c^*)$ yields values for $z_1, \dots, u_m$ which prove that $(a,b,c)$ satisfies $\psi'$. If $c < a < b$ then there exists $\alpha\in \Aut(\fA)$ such that $\alpha(t_{x,1}^*)[x] = a$, $\alpha(t_{x,1}^*)[y] = b$ and $\alpha(t_{x,1}^*)[z_1] > c$. Hence, $\alpha(t_{x,1}^*)$ shows that $(a,b,c)$ satisfies $\psi'$. The argument for $c<b<a$ works with $t_{y,1}^*$ in an analogous way.
\end{proof}

Now we are ready to prove the main result of this subsection.

\begin{prop} \label{thm:mi}
Let $\fA$ be a first-order expansion of $({\mathbb Q};<)$ which is preserved by $\mi$, but not by $\ll$. Then $\J$ has a primitive positive definition in $\fA$. 
\end{prop}
\begin{proof}
If $\leq$ is primitively positively definable in $\fA$, then  Proposition~\ref{thm:mi+leq} yields that $\J$ is primitively positively definable in $\fA$.  
If $\leq$ is not primitively positively definable in $\fA$ then Proposition~\ref{thm:produce-mix} yields that $\fA$ is preserved by $\mix$. In this case Proposition~\ref{thm:mix} implies that $\J$ is primitively positively definable. 
\end{proof}

\subsection{Temporal Structures Preserved  by \texorpdfstring{$\mx$}{mx}} 
\label{sec:mx}
In this section we consider first-order expansions
of $({\mathbb Q};<)$ that are preserved by $\mx$. 
We distinguish the cases whether $X$ is primitively positively definable in $\fA$ or not. Theorem~\ref{thm:mixedmx} implies that if $X$ is not primitively positively definable in $\fA$,
then $\fA$ is also preserved by $\min$.
So we first consider the situation that $\fA$ is preserved by both $\mx$ and $\min$. 
For $R \subseteq \Q^n$, $t=(t_1,\dots, t_n)\in R$ and $I= \{i_1,\dots,i_l\} \subseteq [n]$ we write $\pi_I(t)$ for the tuple $(t_{i_1},\dots,t_{i_l})$ where $i_1 < i_2 < \dots <i_l$ and $\pi_I(R)$ for the relation $\set{\pi_I(t) \mid t\in R}$.

\begin{prop} \label{thm:mxmi}
Let $\fA$ be a first-order expansion 
of $({\mathbb Q};<)$ that is preserved by $\mx$
and $\min$. 
%but $X$ is not primitively positively definable in $\fA$.
Then $\fA$ is preserved by $\mi$. 
\end{prop}
\begin{proof} Let $R$ be a relation in $\fA$. The proof proceeds by induction on the arity $n$ of $R$. For $n=1$, or if $R$ is empty, there is nothing to be shown. Suppose that the statement holds for all relations of arity less than $n$ and that $R$ is not empty. For every $I\subseteq [n]$ we fix a homogeneous system $A^{R}_{I}x=0$ of Boolean linear equations with solution space $\chi_0(\pi_{I}(R))$, which exists due to case~\ref{item:SNF_Mx} in Theorem~\ref{thm:SNF_base}. %t
As $R$ is preserved by $\min$, the Boolean maximum operation preserves $\chi_0(\pi_{I}(R))$.
Furthermore, the solution space of a system of homogeneous linear equations over $\GF_2$ is also preserved by the operation $(x,y,z) \mapsto x+y+z \mod 2$ (because it is a subspace of $\GF_2^3$), we get that $\chi_0(\pi_{I}(R))$ is also preserved by $\min$ because $\min(x,y) = \max(x,y) + x + y \mod 2$. For every pair $t,t'\in R$ we want to show that 
$\mi(t,t')\in R$. If $\min(t)=\min(t')$, we consider the set $S \ceq \set{i\in[n] \mid \chi(t)[i] = \chi(t')[i] = 1}$ and distinguish two cases:
 	 \begin{enumerate}
 	 	\item If $S \neq \emptyset$ then $\chi(\mi(t,t')) = \min(\chi(t),\chi(t'))\in \chi(R)$.
 	 	% $M(t)\cap M(t') \neq \emptyset$, and
 	 	\item If $S = \emptyset$, then $\chi(\mi(t,t'))= \chi(t')\in \chi(R)$.
 	 \end{enumerate}
   \noindent If $\min(t)\neq \min(t')$, then $\chi(\mi(t,t'))\in \{\chi(t),\chi(t')\}\subseteq \chi(R)$. 
   
  Thus, there exists a tuple $c\in R$ with $\chi(c)=\chi(\mi(t,t'))$.  Let $I\ceq \set{i \mid \chi(c)[i] = 1}$ and observe that $I$ is non-empty.  
   By induction hypothesis, the statement holds for $ \pi_{[n]\setminus I}(R) $ and we have $\pi_{[n]\setminus I}(\mi(t,t'))=\mi(\pi_{[n]\setminus I}(t),\pi_{[n]\setminus I}(t'))  \in \pi_{[n]\setminus I}(R)$. Therefore, there exists $r\in R$ with $\pi_{[n]\setminus I}(\mi(t,t'))=\pi_{[n]\setminus I}(r)$. 
   We can apply an automorphism  of $(\Q;<)$  to $r$ to obtain a tuple $r'\in R$ where all entries are positive. 
   We can also apply an automorphism to $c$
   to obtain a tuple $c'\in R$ so that its minimal entries are $0$ and for every other entry $i\in [n]\setminus I$ it holds that $c'[i]>r'[i]$. Then $\mx(c',r')$ yields a tuple in $R$ which is minimal at the coordinates in $I$ and all other coordinates are ordered like the coordinates in $r$, i.e., $\mx(c',r')$ is equal to $\mi(t,t')$ under an automorphism. Hence, $\mi(t,t')\in R$, i.e., $R$ is preserved by $\mi$.
 \end{proof}  

\begin{prop}\label{thm:mx}
Let $\fA$ be a first-order expansion 
of $({\mathbb Q};<)$ that is preserved 
by $\mx$ but not by $\mi$. 
Then $\J$ is primitively positively definable in $\fA$.
\end{prop}
\begin{proof}
First suppose that $X$ is primitively positively definable in $\fA$. 
It is easy to check that $\exists h \big(X(z,z,h) \AND X(x,y,h) \big)$  primitively positively defines $\J(x,y,z)$, and hence $\J$ is primitively positively definable in $\fA$. 
Otherwise, if $X$ is not primitively positively definable in $\fA$, then Theorem~\ref{thm:mixedmx} implies that $\fA$ is also preserved by $\min$,
and hence by $\mi$ by
Proposition~\ref{thm:mxmi}, which contradicts our assumptions. 
\end{proof}

 \subsection{Temporal Structures Preserved  by  \texorpdfstring{$\min$}{min}}\label{sec:min}
This section treats first-order expansions of $({\mathbb Q};<)$ that are preserved by $\min$ but not by 
  $\mi$ and $\mx$. We first show that we may assume that $\leq$ has a primitive positive definition in $\fA$. 
 
 \begin{lem}\label{thm:NoLeqMxOrmi} 
Let $\fA$ be a first-order expansion of $({\mathbb Q};<)$ which is preserved by $\pp$ and 
does not admit a primitive positive definition of $\leq$. Then $\fA$ is preserved by $\mi$ or by $\mx$. 
 \end{lem}
 \begin{proof} By Theorem~\ref{thm:InvPol} there exists an $f\in \Pol(\fA)$ that does not preserve $\leq$. %Then $f$ must be at least binary.
 As $\leq$ is a union of two orbits of $\Aut(\mathbb{Q};<) = \Aut(\fA)$, 
  there is a binary polymorphism $f'$ of $\fA$ that does not preserve $\leq$ by Lemma~10 in~\cite{tcsps-journal}. As $\fA$ is also preserved by $\pp$, Lemma~35 in \cite{tcsps-journal} implies that $\fA$ is preserved by an operation providing \emph{min-intersection closure} or \emph{min-xor closure}. Then $\fA$ is preserved by $\mi$ or by $\mx$ by Proposition~27  and  Proposition~29 in \cite{tcsps-journal}, respectively. 
 \end{proof} 
 
  \begin{prop} \label{thm:min} 
  Let $\fA$ be a first-order expansion of $({\mathbb Q};<)$ preserved by $\min$ but not by 
  $\mi$ and not by $\mx$. 
  Then $R^{\min}_{\leq}$, $R^{\mi}$, and $\J$ have a 
  primitive positive definition in $\fA$. 
 \end{prop}

 \begin{proof} 
 	Let $R$ be a relation of $\fA$ that is not preserved by $\mi$ and let $n$ be the arity of $R$.   As $R$ is preserved by $\min$, it is definable by a conjunction $\phi$ of formulas where each conjunct is of the form as described in Theorem~\ref{thm:SNF_base}~\ref{item:SNF_Min}.  Furthermore, there must be a clause $C$ in $\phi$ that is not preserved by $\mi$. By Theorem~\ref{thm:SNF_base}~\ref{item:SNF_Mi} $C$ is of the form
        \begin{displaymath}
          x>x_{1}\vee \cdots \vee x>x_{\ell} \vee x\geq y_{1}\vee \cdots \vee x \geq y_{k}
        \end{displaymath}
with $k > 1$.
Furthermore, we can assume that $\phi$ is in reduced CNF.
Hence, there exist tuples $t_1, t_2 \in R$ witnessing that the literals $x\geq y_1$ and $x\geq y_2$ cannot be replaced by $x>y_1$ and by $x > y_2$, respectively, i.e.,
\begin{align*}
  t_1 \text{ satisfies } \quad& x=y_1 \AND x<y_2 \AND \bigwedge_{i=1}^\ell x\leq x_i \AND \bigwedge_{i=3}^k x<y_i,\\
  t_2 \text{ satisfies } \quad& x<y_1 \AND x=y_2 \AND \bigwedge_{i=1}^\ell x\leq x_i \AND \bigwedge_{i=3}^k x<y_i.
\end{align*}
%. 
Let $z_{1},\dots,z_{m}$ be all the variables from $\phi$ that do not occur in $C$. Without loss of generality, we may assume that the coordinates of $R$ are in the following order: $x,x_{1},\dots, x_{\ell},y_{1},\dots ,y_{k}$, $z_{1},\dots, z_{m}$.
As $\fA$ is not preserved by $\mx$, Lemma~\ref{thm:NoLeqMxOrmi} implies that $\leq$ has a primitive positive definition in $\fA$;
 so we may assume that $\leq$ is among the relations of $\fA$. 
We claim that $R^{\min}_{\leq}$ can be defined over $\fA$ by the 
primitive positive formula 
$\phi(x,u,v)$ given as follows.

\begin{align}
 \exists z_{1},\dots,z_{m}, x_{1},\dots,x_{\ell}, y_{1},\dots,y_{k} 
 & \Big( R(x,x_{1},\dots, x_{\ell},y_{1},\dots ,y_{k},z_{1},\dots, z_{m}) \nonumber \\
 %  \label{eq:smin} \\
& 
\wedge y_{1}\geq u \AND y_{2}\geq v  \wedge  \bigwedge_{i=1}^\ell x\leq x_{i}   \wedge \bigwedge_{i=3}^k x< y_{i} 
 \Big) \nonumber
\end{align}

To prove the claim, let $(a,b,c)\in R^{\min}_{\leq}$. Assume that $a \geq b$. 
There exists $\alpha\in \Aut(\fA)$ such that $t_1' \ceq \alpha(t_1)$ satisfies $t_1'[x] = a$ and $t_1'[y_2] > \max(a, c)$. 
Now we extend $t_1'$ by two coordinates, named $u$ and $v$ such that $t_1'[u] = b$ and $t_1'[v] = c$. 
Then $\pi_{\set{x,u,v}}(t_1') = (a,b,c)$ and $t_1'$ satisfies the quantifier-free part of $\phi$. Therefore, $\phi(a,b,c)$ holds. 
The case where $a\geq c$ holds is handled analogously using $t_2$ instead of $t_1$.

Now suppose that $(a,b,c)$ satisfies $\phi(x,u,v)$ and let $t^*$ be any tuple which satisfies the quantifier-free part of $\phi$ such that $\pi_{\set{x,u,v}}(t^*) = (a,b,c)$. Then $t^*$ satisfies $C$, and hence $t^*$ satisfies $x\geq y_1 \OR x\geq y_2$. Therefore, $t^* $ satisfies $x\geq u \OR x\geq v$, i.e., $t\in R^{\min}_{\leq}$.
It is easy to check that the formula $\exists h \, (\phi(x,h,y) \wedge h>z)$
is a primitive positive definition of $R^{\mi}$ in $\fA$. Therefore, $\J$ is primitively positively definable in $\fA$ as well (see note below Theorem~\ref{thm:SNF_base}). 
 \end{proof}

\subsection{Definability Dichotomy}
\label{sec:dicho}
In this section we prove Theorem~\ref{thm:dichotomy}, following the strategy outlined earlier, and subsequently we prove Theorem~\ref{thm:BinInjOrJ}

\begin{prop}\label{thm:pp-new}
A temporal relation has a primitive positive definition in $(\Q;T_3)$
if and only if it is preserved by $\pp$. 
\end{prop}
\begin{proof}
By Theorem~\ref{thm:gen-pp}, it suffices to
prove that the relations $\neq$, $R^{\min}_{\leq}$, and
$S^{\mi}$ are primitively positively definable in $(\Q;T_3)$. 
Clearly, $x \leq y$ is equivalent to $\exists z. \, T_3(x,y,z)$ and $x\neq y$ is equivalent to $\exists z.\, T_3(z,x,y)$. 
We claim that the following primitive positive formula defines  $R^{\min}_{\leq}$ in $(\Q;T_3,\leq)$. 
\begin{align*}
\phi(x,y,z) \ceq \exists x',y',z' \big(T_3(x',y',z') \AND x\geq x' \AND y
\leq y' \AND z\leq z'\big)
\end{align*}
Suppose that  $(a,b,c) \in R^{\min}_{\leq}$ holds. By the symmetry of the second and third argument  in $R^{\min}_{\leq}$ we may assume that $a \geq b$ holds. 
Choose $a'=b'$ such that $b\leq a' = b' \leq a$ holds and $c' > \max(a',b',c)$. Then $T_3(a',b',c') \AND a\geq a' \AND b
\leq b' \AND c < c'$ holds and therefore $(a,b,c)$ satisfies $\phi$.  
For the converse direction, suppose for contradiction that $(a,b,c)$ is not in $R^{\min}_{\leq}$ but $\phi(a,b,c)$ holds. Then we
have $a<b \AND a<c$. The quantifier-free part of $\phi$ implies $x'\leq
a < b \leq y'$ and therefore $x' = z' < y'$. However, $c \leq z' = x'
\leq a$ follows, contradicting $a < c$.

Finally, we claim that the formula
\begin{displaymath}
 \psi(x,y,z) \ceq \exists u,v\big(T_3(x,u,v) \wedge (u \neq y) \wedge (v \geq z)\big)
\end{displaymath}
defines $S^{\mi}$. If $(a,b,c)$ satisfies $\psi$ we either have $a=u\neq b$ or $a=v \geq c$. Therefore $(a,b,c)$ satisfies $S^{\mi}$. If $(a,b,c)$ satisfies $S^{\mi}$ we have two cases. If $a\neq b$, we choose $u=a$ and $v>\max(c,a)$. Then $b\neq a=u<v$ and $v>c$ holds and therefore $\psi(a,b,c)$ holds. If $c \leq a$ holds, then we choose $v = a$ and $u>\max(a,b)$. Then $c \leq a=v<u \neq b$ holds, i.e., $\psi(a,b,c)$ holds.
\end{proof}

 \begin{proof}[Proof of Theorem~\ref{thm:dichotomy}]  
 $\Longrightarrow$: Suppose that $\J$ has a primitive positive definition in $\fA$. 
 Then $\fA$ is not preserved by $\ll$ because $\J$ is not preserved by $\lex$: consider for instance $\lex((0,0,1),(2,3,0))$, which is in the same orbit as $(0,1,2)$ and therefore not in $\J$.

 $\Longleftarrow$: Suppose that $\fA$ is not preserved by $\ll$. 
If the relation $T_3$ is primitively positively definable in $\fA$, 
then so is $\J$ by Proposition~\ref{thm:pp-new} because $\J$ is preserved by $\pp$ and we are done. Otherwise,  Theorem~\ref{thm:foundational2} implies 
that $\fA$ is preserved by $\mi$, $\mx$, or $\min$. 
If $\fA$ is preserved by $\mi$, then $\J$ is primitively positively definable in $\fA$ by Proposition~\ref{thm:mi}. 
 If $\fA$ is preserved by $\mx$ but not by $\mi$, 
 then $\J$ is primitively positively definable in $\fA$ by Proposition~\ref{thm:mx}. 
 If $\fA$ is preserved by $\min$ but neither by $\mi$ nor by $\mx$, then $\fA$ primitively positively defines $\J$ by Proposition~\ref{thm:min}.
 \end{proof}

\begin{proof}[Proof of Theorem~\ref{thm:BinInjOrJ}]
%If $\ll$ is polymorphism of $\fA$, then $\fA$ has a binary injective polymorphism.
Suppose that $\fA$ does not have a binary injective polymorphism.
Then $\fA$ is preserved by $\min, \mi$, $\mx$, or their duals. Therefore, $\fA$ is preserved by $\pp$ or $\dpp$ by the inclusions presented in Section~\ref{sec:polysOfTemporal}.
If $\fA$ is preserved by $\pp$, then Theorem~\ref{thm:dichotomy} implies that $\J$ is primitively positively definable in $\fA$. 
If $\fA$ is preserved by $\dpp$, the dual of $\fA$, i.e., the structure obained from $\fA$ by substituting all relations by their duals, has $\pp$ as a polymorphism. Hence, $\J$ has a primitive positive definition in the dual of $\fA$ and therefore $-\J$ has a primitive positive definition in $\fA$.

It remains to show that 
 the two cases of the theorem are mutually exclusive. 
Suppose that $\fA$ has a binary injective polymorphism $f$; 
% $f\colon \Q^2 \rightarrow \Q$ be a binary injective operation and 
we may also assume without loss of generality
that $f(0,1) > f(1,0)$. 
Since $f(0,1) \neq f(0,2)$ we have $(f(0,1),f(0,2),f(1,0)) \notin \J$.
As $(0,0,1),(1,2,0) \in \J$ we have that 
$\J$ is not preserved by $f$ and hence does not have a primitive positive definition in $\fA$. The dual case works analogously. Therefore,  primitive positive definability of $\J$ in $\fA$ and binary injective polymorphisms in $\Pol(\fA)$ are mutually exclusive by Theorem~\ref{thm:InvPol}. 
\end{proof}

\section{Combinations of Temporal CSPs}
\label{sec:combine}
In this section we prove that the every generic combination of the structure $({\mathbb Q};<,\J)$ with another structure that can prevent crosses has an NP-hard CSP (Theorem~\ref{thm:JPreventCrossesHard}). 
We then derive our complexity classification for
the CSP of combinations of temporal structures (Theorem~\ref{thm:dichotomyExpansionsQ}). 
In our NP-hardness proof we use the following.

\begin{prop}[Corollary 6.1.23 in~\cite{Book}]\label{thm:ecsp-hard}
Let $\fA$ be a countably infinite $\omega$-categorical structure with finite relational signature and without constant polymorphisms. If all polymorphisms of $\fA$ are essentially unary then $\CSP(\fA)$ is NP-hard.
\end{prop}

The next definition introduces the key property of the polymorphisms of $(\Q;\J)$.

\begin{defi}
For any $n,i\in \N$, $1\leq i\leq n$, $a\in \Q^n$,  and operations $f\colon \Q^n \rightarrow \Q$ we define
\begin{align*}
%  x < y &:\Leftrightarrow  \forall 1\leq k\leq n\colon y_k < x_k \quad(\text{and likewise for } \leq)\\
  H(a,i)& \coloneqq \set{b\in \Q^n \mid \text{for all } j \in [n] \setminus \{i\} \text{ we have } b_j > a_j \text{ and } b_i = a_i}\\
  \text{and} \quad I_f(a) & \coloneqq \set{i\in \N \mid f \text{ is constant on } H(a,i)}.
\end{align*}
Let $\cK$ be the set of all operations $f\colon \Q^{n}  \rightarrow \Q$ with $n \geq 1$ where $I_f(a) \neq \emptyset$ for all $a\in \Q^{n}$. %It is worth while to visualize $H(x,i)$ for $n=2$ and $n=3$ and to draw some binary function in $\cK$ yourself, in particular one that preservers $<$.
\end{defi}

Examples of operations in $\cK$ are $\min$, $\mi$, $\mix$, $\mx$, $\pp$, and all unary operations. 
Non-examples are $\max$ and $\ll$.

\begin{lem}\label{thm:PolJInK}
All polymorphisms of $(\Q; \J)$ are in $\cK$.
\end{lem}

\begin{proof}
Let $f\colon \Q^n \rightarrow \Q$ be a polymorphism of $(\Q;\J)$. We proceed by induction on $n \in {\mathbb N}$. If $n=1$, the statement is trivial. 
For $n\geq 2$, assume towards a contradiction that $f \not \in \cK$. Then there exists $c\in \Q^n$ such that for every $k \in [n]$ there exists $a^k, b^k \in H(c, k)$ such that $f(a^k) < f(b^k)$. 
Without loss of generality we may assume that $\max(f(b^1),\dots, f(b^n)) = f(b^1)$.
If there exists $k\neq 1$ and $e > b^1_1$ such that $f(a^k) \neq f(e,a^k_2,\dots,a^k_n)$, then  $(a^k_1, e, b^1_1) \in \J$  and $(a^k_l,a^k_l,b^1_l) \in \J$ for all $l \in \set{2,\dots,n}$, but $(f(a^k), f(e,a^k_2,\dots,a^k_n), f(b^1)) \not\in \J$ because $f(b^1) \geq f(a^k) \neq f(e,a^k_2,\dots,a^k_n)$, contradicting the assumption that $f$  preserves $\J$. 
%Hence, for every $k\neq 1$ and every $e > c_1=b^1_1$  we have $f(e, a^k_2, \dots, a^k_n) = f(a^k)$. 
Similarly, if there exists $k\neq 1$ and $e > b^1_1$ such that $f(b^k) \neq f(e,b^k_2,\dots,b^k_n)$, then  $(b^k_1, e, b^1_1) \in \J$  and $(b^k_l,b^k_l,b^1_l) \in \J$ for all $l \in \{2,\dots,n\}$, but $(f(b^k), f(e,b^k_2,\dots,b^k_n), f(b^1)) \not\in \J$.
 Hence, for every $k\neq 1$ and every $e > b^1_1$  we have 
\begin{displaymath}
  f(e, a^k_2, \dots, a^k_n) = f(a^k) <  f(b^k) = f(e, b^k_2, \dots, b^k_n).
\end{displaymath}
Choose $e>b^1_1$ and define $f' \colon \Q^{n-1} \to \Q$ as $(x_2,\dots,x_n) \mapsto f(e, x_2, \dots, x_n)$; as $\J$ is preserved by all constant polymorphisms, $f'$ is a composition of polymorphisms of $(\Q;\J)$ and hence a polymorphism of $(\Q;\J)$. 
Then for all $k\in \set{2,\dots,n}$ we have 
\begin{align*}
 (b^k_2, \dots, b^k_n),(a^k_2, \dots, a^k_n) &\in H((c_2, \dots, c_n),k) \quad \text{and} \\
 f'(a^k_2, \dots, a^k_n) = f(e,a^k_2, \dots, a^k_n) & < f(e,b^k_2, \dots, b^k_n) =  f'(b^k_2, \dots, b^k_n).
\end{align*}
Therefore, $f'$ is an $(n-1)$-ary polymorphism of $\fA$ which is not in $\cK$, a contradiction to the induction hypothesis. 
\end{proof}

\begin{lem}\label{thm:backwardPropagationOfRays}
Let $f\in \cK \cap \Pol(\Q;<)$ be of arity $n\geq 2$.
Let $a,b \in \Q^n$ and $i\in I_f(a)$. 
\begin{enumerate}
\item If $b_i < a_i \AND \bigwedge_{k\neq i} b_k > a_k$, then $I_f(b) = \set{i}$.
\label{case:strictlyBackwards}
\item If $b_i \leq  a_i \AND \bigwedge_{k\neq i} b_k \geq a_k$, then $i\in I_f(b)$. \label{case:nonStrictlyBackwards}
\end{enumerate}
  \end{lem}
\begin{proof}
To prove (1), suppose for contradiction that 
$b_i < a_i \AND \bigwedge_{k\neq i} b_k > a_k$ and $j\in I_f(b)$ with $j\neq i$. Then there exists $c \in H(b,j)$ such that $b_i < c_i < a_i$. Now consider $d,e \in \Q^n$ such that 
  \begin{align*}
    &  d_j = c_j \AND d_i = a_i \AND \bigwedge_{k\not\in {i,j}} d_k > c_k \quad \text{ and } \quad 
     e_i = a_i \AND \bigwedge_{k\neq i} e_k > d_k.
  \end{align*}
  Then $e_j > d_j = c_j = b_j > a_j$, $d_i = a_i = e_i > c_i > b_i$, and $e_k > d_k > c_k > b_k > a_k$ for $k\in [n] \setminus \{i,j\}$. Hence, $d\in H(b,j) \cap H(a,i)$ and $e\in H(a,i)$. This implies that $f(c) = f(d) = f(e)$, which contradicts
  the assumption that $f$ preserves $<$, because $c_k < e_k$ for every $k \in [n]$. 
  Since $I_f(b) \neq \emptyset$ by assumption, we therefore conclude that 
  $I_f(b) = \{i\}$. 
  
To prove (2), first consider the case that $b_i = a_i$. Then $H(b,i) \subseteq H(a,i)$ and therefore $i\in I_f(b)$. 
  If $b_i<a_i$, choose $u,v\in H(b,i)$. Then there exists  $b'\in H(b,i)$ such that for each $k\neq i$ we have $b'_k < \min(u_k, v_k)$. Then $u,v\in H(b',i)$ and $b'$ satisfies $b'_i < a_i \AND \bigwedge_{k\neq i} b'_k > a_j$. Hence, we have $I_f(b') = \{i \}$ by the first claim of the statement and therefore $f(u) = f(v)$. As $H(b',i) \subseteq H(b,i)$ we conclude that $i\in I_f(b)$. 
\end{proof}

\begin{proof}[Proof of Theorem~\ref{thm:JPreventCrossesHard}]
 Let $\fB$ be the generic combination of $({\mathbb Q};<,\J)$ and $\fA$, which exists by Theorem~\ref{thm:existenceGenericComb}. Without loss of generality we may assume that the domain of $\fB$ is $\Q$ and that $\fA$ and $({\mathbb Q};<,\J)$ are reducts of $\fB$. 
  Let $f \in \Pol(\fB)$ be of arity $n$. 
 Our goal is to show that $f$ is  essentially unary; the NP-hardness of $\CSP(\fB)$ then follows from Proposition~\ref{thm:ecsp-hard}.

  By Lemma~\ref{thm:PolJInK} we have 
  $\Pol(\fB) \subseteq \Pol(\Q;<,\J) \subseteq \cK$ and therefore $f\in \cK$. 
  Suppose for contradiction that there are $a,b\in \Q^n$ such that $i\in I_f(a)$, $j\in I_f(b)$, and $i\neq j$. We will treat the case that $i=1$ and $j=2$; all other cases can be treated analogously.
  Let $\phi$ be a cross prevention formula of $\fA$. 
 %\red{NEW PROOF: 
 %Since $\fB$ is a generic combination,
 %there exists a tuple $a',b'$ such that 
 %$a'$ lies in the same $\Aut(\bA)$-orbit as $a$
 %and $b'$ lies in the same $\Aut(\bA)$-orbit as $b$  and $\phi(a_
 %}

Consider the following first-order formula
   $\psi(\bar x, \bar y,\bar u,\bar v)$ with parameters $a_1, \dots, a_n$, $b_1, \dots, b_n$.
  \begin{align*}
    \psi  \ceq\quad  & \; x_1 < a_1 \AND x_1 = y_1  \AND \bigwedge_{k\in [n]\setminus\set{1}} (x_k > a_k \AND y_k > a_k) \\
     \AND  & \; u_2 < b_2 \AND u_2=v_2 \AND \bigwedge_{k\in [n]\setminus\set{2}} (u_k > b_k \AND v_k > b_k)
  \end{align*}
For $k\in [n]$ let $\psi^{k}(x_k, y_k, u_k, v_k)$ be the conjunction of all atomic formulas in $\psi$ that contain $x_k$, $y_k$, $u_k$, or $v_k$. 
  Notice that every atomic formula in $\psi$ only contains variables from $\{x_k, y_k, u_k, v_k\}$ for a fixed $k$. 
  Hence, $\psi(\bar x, \bar y,\bar u,\bar v)$ is equivalent to $\bigwedge_{k=1}^n \psi^{k}(x_k, y_k, u_k,v_k)$.
  Let $\delta(z_1, z_2, z_3, z_4)$ be the first-order formula $$z_1 = z_2 \wedge z_2 \neq z_3 \wedge z_3 \neq z_4 \wedge z_2 \neq x_4.$$
  % $$z_1 = z_2 \AND \bigwedge_{1<i<j} z_i \neq z_j.$$ 
For each $k$ there exists an assignment 
$s_{k,1} \colon \{x_k,y_k,u_k,v_k\} \to \Q$ which satisfies $\psi^k$ and additionally satisfies 
$\delta(x_k,y_k,u_k,v_k) \vee \delta(u_k,v_k,x_k,y_k)$. 
%$$(\Q;<) \models \psi^k(s_{k,1}(x_k), s_{k,1}(y_k), s_{k,1}(u_k), s_{k,1}(v_k))$$ and 
%$$(\Q;<) \models \delta(s_{k,1}(x_k), s_{k,1}(y_k), s_{k,1}(u_k), s_{k,1}(v_k)) \vee \delta(s_{k,1}(u_k), s_{k,1}(v_k), s_{k,1}(x_k), s_{k,1}(y_k)).$$ 
For $k=1$ there exists an assignment $s_{1,2}$
that satisfies $\phi(x_k,y_k,u_k,v_k) \wedge \delta(x_k,y_k,u_k,v_k)$, and 
  for each $k>1$ there exists an assignment $s_{k,2}$ that satisfies
  $\phi(x_k,y_k,u_k,v_k) \wedge \delta(u_k,v_k,x_k,y_k)$. 
  %such that $\fA$ satisfies 
 % $\phi(s_{k,2}(x_k), s_{k,2}(y_k), s_{k,2}(u_k), s_{k,2}(v_k)) \; \AND \; \delta(s_{k,2}(u_k), s_2(v_k), s_{k,2}(x_k), s_{k,2}(y_k))$. 
For each $k$, both $s_{k,1}$ and $s_{k,2}$ can be chosen such that their images are disjoint to the set of all entries of $c \ceq (a_1, \dots, a_n, b_1, \dots, b_n)$, because both $(\Q;<)$ and $\fA$ do not have algebraicity.
  Now we apply the first statement of Lemma~2.7 in~\cite{BodirskyGreinerCombinations}  for each $k\in [n]$ to $c$, 
  $s_{k,1}$, $s_{k,2}$. 
  %$t_{k,1} = (s_{k,1}(x_k), s_{k,1}(y_k), s_{k,1}(u_k), s_{k,1}(v_k))$ and $t_{k,2} := (s_{k,2}(x_k), s_{k,2}(y_k), s_{k,2}(u_k), s_{k,2}(v_k))$. 
  This yields, for each $k\in [n]$, 
  a solution $s^k$ to $\psi^k \wedge \phi(x_k,y_k,u_k,v_k)$. 
 %  the existence of a tuple $t_k = (s(x_k), s(y_k), s(u_k), s(v_k))$ in $\fB$ which is in the same $\Aut(\Q;<,c)$-orbit as $t_{k,1}$ and in the same $\Aut(\fA,c)$-orbit as $t_{k,2}$. Therefore, $\psi^k(t_k) \AND \phi(t_k)$ holds in $\fB$. 
 % Hence, $\psi \AND \bigwedge_{k=1}^n \phi(x_k,y_k,u_k,v_k)$ is satisfiable in $\fB$.
  % and $s$ is a satisfying assignment.
  Let $s(x)$ denote $(s^1(x_1), \dots, s^n(x_n))$ and likewise for $s(y), s(u), s(v)$.
  Let $a'$ and $b'$ be the componentwise minimum of $a, s(x),s(y)$ and $b,s(u),s(v)$, respectively. Then $s(x),s(y) \in H(a',1)$
  and $s(u),s(v) \in H(b',2)$.
 We apply Case \ref{case:nonStrictlyBackwards} of Lemma~\ref{thm:backwardPropagationOfRays} to $a$ and $a'$ (in the role of $b$) and $i=1$ and get $1\in I_f(a')$. Similarly, we apply 
   Case \ref{case:nonStrictlyBackwards} of Lemma~\ref{thm:backwardPropagationOfRays} to $b$, $b'$ and $i=2$ and get $2\in I_f(b')$. 
  Therefore, $f(s(x)) = f(s(y))$ and $f(s(u))=f(s(v))$ must hold. 
  However, as $f$ preserves $\phi$ we must also have $\phi(f(s(x)), f(s(y)), f(s(u)), f(s(v)))$, contradicting the fact that $\phi(x,y,u,v) \AND x=y \AND u=v$ is not satisfiable in $\fA$. 
    
We conclude that there exists an $i \in [n]$ such that $I_f(a) = \{i\}$ for all $a\in \Q^n$. This implies that $f$ only depends on the $i$-th coordinate: to prove this, let $a,b\in \Q^n$ be such that $a_i=b_i$. We choose any $c \in \Q^n$ such that $c_i=a_i$ and $c_j <\min(a_{j},b_{j})$ for every $j\in [n]\setminus \{i\}$.
As $a,b\in H(c,i)$ and $i\in I_{f}(c)$ we have $f(a)=f(b)$, i.e., $f$ can only depend on the $i$-th coordinate. The case that $f$ is constant cannot happen, because $f$ preserves $<$.
Thus, $f$ is essentially unary. 
\end{proof}

Theorem~\ref{thm:JPreventCrossesHard} is applicable to countably infinite $\omega$-categorical structures with finite relational signature which can prevent crosses and do not have algebraicity. 
Besides $(\Q;<)$, the following structures satisfy all of these conditions:
\begin{itemize}
 \item the random graph with edge and non-edge relation~\cite{BodPin-Schaefer-both}
 \item the univeral homogeneous $K_n$-free graph, for $n \geq 3$, also called Henson graph~\cite{BMPP16} with edge relation
  \item first-order expansions of the binary branching $C$-relation in~\cite{Phylo-Complexity}
 \item the \Fresse -limit of all finite 3-uniform hypergraphs which do not embed a tetrahedron (see Chapter~6 in~\cite{Hodges} for the construction method)
\end{itemize}

\begin{proof}[Proof of Theorem~\ref{thm:dichotomyExpansionsQ}]
Let $\fB$ be the generic combination of $\fA_1$ and $\fA_2$, which exists by Theorem~\ref{thm:existenceGenericComb}.
We may assume that 
$\fB$, $\fA_1$, and $\fA_2$ all have the domain $\Q$ and that $\fA_1$ and $\fA_2$ are reducts of $\fB$. For $i=1$ and $i=2$, 
let $<_i$ be a linear order on $\Q$ such that all relations of $\fA_i$ are first-order definable in $(\Q;<_i)$; correspondingly $\Betw_i,\Cycl_i,\Sep_i,\J_i$ are defined as the relations $\Betw,\Cycl,\Sep,\J$ but with respect to $<_i$ instead of $<$. The same holds for $\min_i, \mi_i, \mx_i, \ll_i$ and their duals.

If both $\fA_1$ and $\fA_2$ have a constant polymorphism, then both $\fA_1$ and $\fA_2$ have all constant operations as polymorphisms, and it follows that $\fB$ has a constant polymorphism, too. In this case $\CSP(\fB) = \CSP(T_1 \cup T_2)$ can be solved in constant time because only instances with an empty relation or $\bot$ as conjunct are unsatisfiable (item (2) of the statement). Hence, we may suppose without loss of generality that $\fA_1$ does not have a constant polymorphism. Then by Theorem~\ref{thm:end}, one of the following cases applies.
\begin{itemize}
\item $\fA_1$ is preserved by all permutations;
\item the relation $\Betw_1$, $\Cycl_1$, or $\Sep_1$ is primitively positively definable in $\fA_1$;
\item the relation $<_1$ is primitively positively definable in $\fA_1$.
\end{itemize}
In the first case, $\fB$ itself is a temporal structure 
%and we are in the second case of the statement of the theorem. 
and $\CSP(\fB)$ is in P (item (3) of the statement) or NP-complete by Theorem~\ref{thm:bod-kara}. 
If one of the relations  $\Betw_1$, $\Cycl_1$, or $\Sep_1$ is primitively positively definable in $\fA_1$, then $\CSP(\fA_1)$ is NP-hard
 and hence
$\CSP(\fB)$ is NP-hard.
So we may assume in the following that $<_1$ is primitively positively definable in $\fA_1$. Hence, we can assume without loss of generality that $<_1$ is in the signature of $\fA_1$.

We now consider the case that both $\fA_1$ and $\fA_2$ have a binary injective polymorphism.
If for some $i \in \{1,2\}$ the problem $\CSP(\fA_i)$ is NP-hard, then clearly $\CSP(\fB) = \CSP(T_1 \cup T_2)$ is NP-hard as well. 
Otherwise, Theorem~\ref{thm:bod-kara} implies that for $i=1$ and $i=2$, the structure $\fA_i$ is preserved by $\ll$ or by $\dll$ (or P=NP, in which case Theorem~\ref{thm:dichotomyExpansionsQ} is trivial). Note that $\ll$ and $\dll$ also preserve $\neq$, so we may add $\neq$ to the signature of $\fA_1$ and $\fA_2$. 
As $\neq$ is independent from $T_1$ and from $T_2$ by Proposition~\ref{thm:indep}, the polynomial-time tractability of $\CSP(T_1 \cup T_2)$ then follows from Theorem~\ref{thm:P}.

If $\fA_1$ does not have a binary injective polymorphism, then $\CSP(\fA_1)$ and $\CSP(\fB)$ are NP-hard unless 
$\mx_1$, $\min_1$, $\mi_1$, or one of their duals is a polymorphism of $\fA_1$, by Theorem~\ref{thm:bod-kara}. 
We assume in the following that $\fA_1$ is preserved by $\mx_1$, $\min_1$, or $\mi_1$; if $\fA_1$ is preserved by one of their duals, then the NP-hardness of $\CSP(T_1 \cup T_2)$  can be shown analogously.
By Theorem~\ref{thm:BinInjOrJ}, the relation $\J_1$  has a primitive positive definition in $\fA_1$.

Now, we make a case distinction for $\fA_2$.
 If the structure $\fA_2$ is preserved by all permutations, we are again done (this is analogous to the situation that $\fA_1$ is preserved by all permutations, which was already treated above). 
Otherwise, we apply Theorem~\ref{thm:end} to $(\fA_2;\neq)$ and 
obtain that a relation  $R \in \{<_2,\Betw_2,\Cycl_2,\Sep_2\}$ has a 
primitive positive definition $\phi$ in $(\fA_2; \neq)$. Let $E$ be the set of all sets $\set{x_i, x_j}$ such that $x_i \neq x_j$ appears in $\phi$ and $n$ the arity of $R$. Then, for some $m\geq n$, the formula $\phi$ can be written in the following way:
$\phi(x_1, \dots, x_n) = \exists x_{n+1}, \dots, x_{m} \big(\psi(x_1, \dots, x_m) \AND \bigwedge_{{i,j}\in E} x_i \neq x_j\big)$
where $\psi$ is a primitive positive $\tau_2$-formula. 
Notice that for all $i,j\in [n] $ with $i\neq j$ we may add $\set{i,j}$ to $E$ because for any choice of $R$, all coordinates in tuples of $R$ are pairwise distinct.

Consider the undirected graph  $([m], E)$. We may choose any linear order $E_d'$ on $[n]$ and extend $E_d'$ to $E_d$ on $[m]$ by choosing a direction for each edge in $E$ such that $([m], E_d)$ is a cycle-free directed graph.
Because $x<y \OR x>y$ defines $x\neq y$ we have 

\begin{align*}
\phi(x_1, \dots, x_n) \equiv \exists x_{n+1}, \dots, x_{m} \Big(&(\psi(x_1, \dots, x_m) \AND \bigwedge_{(i,j) \in E_d} ((x_i <_1 x_j) \OR (x_j <_1 x_i))\Big).
\end{align*}

Now notice that 
$\exists x_{n+1}, \dots, x_{m} \big(\psi(x_1, \dots, x_m) \AND \bigwedge_{(i,j) \in E_d} x_i <_1 x_j\big)$ 
is a primitive positive formula in $\fB$ which defines the same relation as the formula
\begin{equation}\label{eq:shortNotationREd}
 R(x_1, \dots, x_n) \AND \bigwedge_{(i,j) \in E_d'} x_i <_1 x_j
\end{equation}
in $\fB$.
Now, we go through all possible choices for $R$ and present primitive positive definitions for either $<_2$ or $\neq$ in $\fB$.
In order to simplify the presentation, we will use conjuncts of the form~\eqref{eq:shortNotationREd} instead of their equivalent primitive positive definitions in $\fB$.

\begin{itemize}
 \item If $R$ equals $<_2$ then $\exists z\big((x<_2 z \AND x <_1 z) \AND (z <_2 y \AND y<_1 z))$ is a primitive positive definition of $x <_2 y$. This is easy to see with the equivalent expression $\exists z \big((x <_2 z <_2 y) \AND (x <_1 z) \AND (y <_1 z)\big)$.
 \item If $R$ equals $\Betw_2$ we claim that 
 \begin{align*}
  \exists u,v \big(&(\Betw_2(x,u,v) \AND (x <_1 u) \AND (u <_1 v))\\
   \AND\; &(\Betw_2(u,v,y) \AND (y <_1 u) \AND (u <_1 v))\big)
 \end{align*}
is a primitive positive definition of $x\neq y$. Again we give an equivalent expression which helps to verify the claim:
$\exists u,v \big(((x<_2 u <_2 v <_2 y) \OR (y <_2 v <_2 u <_2 x)) \AND (x <_1 u <_1 v) \AND (y <_1 u)\big)$.  In the latter, it is clear that $x\neq y$ always holds and that all distinct $x,y$ satisfy the formula.
 
 \item If $R$ equals $\Cycl_2$ we claim that 
  \begin{align*}
   \exists u,v \big(&(\Cycl_2(x,u,v) \AND (x <_1 u) \AND (u <_1 v)) \\ 
   \AND\;&(\Cycl_2(u,y,v) \AND (y <_1 u) \AND (u <_1 v))\big)
   \end{align*}
   is a primitive positive definition of $x\neq y$. A case analysis of $\Cycl_2(x,u,v)$ yields that the given formula is equivalent to
 \begin{align*}
  \exists u,v \big(\big(&\quad (x <_2 u <_2 y <_2 v)\\
  & \OR (u <_2 y <_2 v <_2 x)\\
  & \OR (y <_2 v <_2 x <_2 u) \\
  & \OR (v <_2 x <_2 u <_2 y) \big)\\
    \AND\; & (x<_1 u <_1 v) \AND (y<_1 u)\big).
 \end{align*}
In the latter formula, it is clear that $x\neq y$ must always hold and that for any distinct $x,y$ the formula is satisfiable.
 \item If $R$ equals $\Sep_2$ we claim that 
 \begin{align*}
  \exists u,v, w \big(&(\Sep_2(x,u,v,w) \AND (x <_1 u) \AND (u <_1 v) \AND (v <_1 w)) \\
  \AND\; & (\Sep_2(u,v, w,y) \AND (y <_1 u) \AND (u <_1 v) \AND (v<_1 w))\big)
  \end{align*}
is a primitive positive definition of $x\neq y$. Similarly to above, a case analysis of $\Sep_2$ yields an equivalent expression
\begin{align*}
  \exists u,v,w \big(\big(&\quad (x <_2 v <_2 y <_2 u <_2 w) \\
  &\OR (x <_2 w <_2 u <_2 y <_2 v) \\
  &\OR (u <_2 y <_2 v <_2 x <_2 w)\\
  &\OR (u <_2 w <_2 x <_2 v <_2 y) \OR  (y <_2 u <_2 w <_2 x <_2 v)\\
  &\OR (v <_2 x <_2 w <_2 y <_2 u) \\
  &\OR (v <_2 u <_2 y <_2 w <_2 x) \\
  &\OR (w <_2 y <_2 u <_2 v <_2 x) \\  
  &\OR (w <_2 x <_2 v <_2 u <_2 y) \OR (y <_2 w <_2 x <_2 v <_2 u) \big)\\
  \AND\; & (x <_1 u <_1 v <_1 w) \AND (y <_1 u)\big)
 \end{align*}
 for which the claim is easily verified because $x\neq y$ always holds and for any distinct $x,y$ there exist $u,v,w$ satisfying the formula.
\end{itemize}

Choose a relation $S$ from $\set{\neq, <_2}$ which is primitively positively definable in $\fB$ and let $\fA'_2$ be the expansion of $\fA_2$ by $S$ and $\fB' \ceq \fA_1 \ast \fA_2'$. As $S$ is primitively positively definable in $\fB$, it suffices to show NP-hardness of $\CSP(\fB')$ instead of $\CSP(\fB) = \CSP(T_1 \cup T_2)$. 

If $S$ is $<_2$, then $\fA_2'$ has cross prevention, so the NP-hardness of $\CSP(\fB')$ follows from Theorem~\ref{thm:JPreventCrossesHard}. 
If $S$ is $\neq$, then we may again apply Theorem~\ref{thm:end} to $\fA_2'$ to conclude that the relation $<_2$, $\Betw_2$, $\Cycl_2$, or $\Sep_2$ is primitively positively definable in $\fA'_2$.
The first case has already been treated above. 
In the remaining cases we get NP-hardness of $\CSP(\fA_2')$ and hence of $\CSP(\fB')$. 
\end{proof}

\section{Conclusion and Outlook}
Our results show that there are two temporal relations, namely $\J$ and its dual, with the property that every first-order expansion of $(\Q;<)$ where the weakened Nelson-Oppen conditions do not apply, i.e., $\neq$ is not independent from their theory, can define one of these relations primitively positively. 
We also showed that $\CSP({\Th(\Q;\J,<)} \cup \Th(\fA))$ is NP-hard for structures $\fA$ that satisfy the fairly weak assumption of cross prevention and have a generic combination with $(\Q;<)$. These results can be used to prove a complexity dichotomy for combinations of temporal CSPs: they are either in P or NP-complete. 
Our results also motivate the following conjecture, which remains open in general. 

\begin{conj}
Let $\fA_1$ and $\fA_2$ be countably infinite $\omega$-categorical structures without algebraicity that are not preserved by all permutations and that have the cross prevention property. 
If 
\begin{itemize}
\item $\CSP(\fA_i)$ is in P and $\fA_i$ has a binary injective polymorphism for both $i = 1$ and $i=2$, or
\item $\fA_i$ has a constant polymorphism for both $i = 1$ and $i=2$, 
\end{itemize}
then $\CSP(\Th(\fA_1) \cup \Th(\fA_2))$ is in P. 
Otherwise, $\CSP(\Th(\fA_1) \cup \Th(\fA_2))$ is NP-hard. 
 \end{conj}

%\printbibliography  
 %\clearpage
%\newpage 

\bibliographystyle{alphaurl}
\bibliography{global.bib}

 \end{document}